\documentclass[pagesize,pdftex]{scrartcl}
\usepackage{latexsym} 
\usepackage[intlimits]{amsmath}
\usepackage{amsthm}
\usepackage{amsfonts}
\usepackage{amssymb}
\usepackage{amsxtra}
\usepackage{amscd}
\usepackage{ifthen}
\usepackage{graphicx}
\usepackage[shortlabels]{enumitem}
\usepackage{mathrsfs}
\usepackage[pagebackref=true]{hyperref}

\hypersetup{
pdfauthor={Giovanni P. Galdi and Mads Kyed},
pdftitle={Time-periodic solutions to the Navier-Stokes equations in the three-dimensional whole-space with a drift term: Asymptotic profile at spatial infinity},
breaklinks=true,
colorlinks=true,
linkcolor=blue,
citecolor=blue,
urlcolor=blue,
filecolor=blue,
}

\numberwithin{equation}{section} 
\setkomafont{title}{\normalfont}

\newenvironment{pdeq}{ \left\{ \begin{aligned}}{\end{aligned}\right.}

\newcommand{\np}[1]{(#1)}
\newcommand{\nb}[1]{[#1]}
\newcommand{\bp}[1]{\big(#1\big)}
\newcommand{\bb}[1]{\big[#1\big]}
\newcommand{\Bp}[1]{\bigg(#1\bigg)}
\newcommand{\Bb}[1]{\bigg[#1\bigg]}
\newcommand{\Bbp}[1]{\bigg[#1\bigg)}
\newcommand{\calg}{{\mathcal G}}

\newcommand{\calj}{{\mathcal J}}
\newcommand{\calk}{{\mathcal K}}

\newcommand{\calp}{{\mathcal P}}

\newcommand{\calt}{{\mathcal T}}

\newcommand{\R}{\mathbb{R}}
\newcommand{\Z}{\mathbb{Z}}
\newcommand{\CNumbers}{\mathbb{C}}
\newcommand{\N}{\mathbb{N}}
\DeclareMathOperator{\e}{e}

\DeclareMathOperator{\id}{Id}
\DeclareMathOperator{\Div}{div}

\DeclareMathOperator{\supp}{supp}

\DeclareMathOperator{\realpart}{Re}

\newcommand{\Ra}{\Rightarrow}

\newcommand{\ra}{\rightarrow}

\newcommand{\set}[1]{\ensuremath{\{#1\}}}

\newcommand{\setc}[2]{\ensuremath{\{#1\ \lvert\ #2\}}}
\newcommand{\setcl}[2]{\ensuremath{\bigl\{#1\ \lvert\ #2\bigr\}}}

\newcommand{\closure}[2]{\overline{#1}^{#2}}
\newcommand{\seqK}[1]{\ensuremath{\set{#1}_{k=1}^\infty}}

\newcommand{\proj}{\calp}
\newcommand{\projsymbol}{{\delta_{\Z}}}
\newcommand{\projcompl}{\calp_\bot}
\newcommand{\hproj}{\calp_H}
\newcommand{\quotientmap}{\pi}
\newcommand{\grp}{G}
\newcommand{\dualgrp}{\widehat{G}}

\newcommand{\Torus}{\R/\per\Z}
\newcommand{\idmatrix}{I}
\newcommand{\B}{B}
\newcommand{\RthreeR}{\R\times\R^3}
\newcommand{\grad}{\nabla}

\newcommand{\dx}{{\mathrm d}x}
\newcommand{\dr}{{\mathrm d}r}
\newcommand{\dtau}{{\mathrm d}\tau}

\newcommand{\dg}{{\mathrm d}g}
\newcommand{\ds}{{\mathrm d}s}
\newcommand{\dt}{{\mathrm d}t}

\newcommand{\dy}{{\mathrm d}y}

\newcommand{\dS}{{\mathrm d}S}

\newcommand{\SR}{\mathscr{S}}

\newcommand{\TDR}{\mathscr{S^\prime}}

\newcommand{\ft}[1]{\widehat{#1}}

\newcommand{\FT}{\mathscr{F}}
\newcommand{\iFT}{\mathscr{F}^{-1}}
\newcommand{\riesztrans}{\mathfrak{R}}
\newcommand{\norm}[1]{\lVert#1\rVert}
\newcommand{\oseennorm}[2]{\norm{#1}_{#2,\mathrm{Oseen}}}

\newcommand{\snorm}[1]{{\lvert #1 \rvert}}
\newcommand{\snorml}[1]{{\bigl\lvert #1 \big\rvert}}
\newcommand{\snormL}[1]{{\Bigl\lvert #1 \Big\rvert}}
\newcommand{\snormLL}[1]{{\Biggl\lvert #1 \Bigg\rvert}}
\newcommand{\labs}[1]{\big| #1 \big|}

\newcommand{\WSR}[2]{{W^{#1,#2}}} 
 
\newcommand{\DSR}[2]{D^{#1,#2}} 
\newcommand{\DSRN}[2]{D^{#1,#2}_0} 

\newcommand{\CR}[1]{C^{#1}}  
\newcommand{\LR}[1]{L^{#1}}
 
\newcommand{\lR}[1]{\ell^{#1}}

\newcommand{\LRloc}[1]{L^{#1}_{loc}} 
\newcommand{\CRi}{\CR \infty}
\newcommand{\CRci}{\CR \infty_0}

\newcommand{\DSRNsigma}[2]{D^{#1,#2}_{0,\sigma}}

\newcommand{\CRcisigma}{\CR{\infty}_{0,\sigma}}

\newcommand{\WSRper}[2]{W^{#1,#2}_{\mathrm{per}}} 
\newcommand{\CRper}{\CR{\infty}_{\mathrm{per}}}

\newcommand{\CRciper}{\CR{\infty}_{0,\mathrm{per}}}
\newcommand{\CRcisigmaper}{\CR{\infty}_{0,\sigma,\mathrm{per}}}

\newcommand{\xoseen}[1]{{X}^{#1}_{\mathrm{Oseen}}}

\newcommand{\xpresper}[1]{{Y}^{#1}_{\mathrm{per}}}

\newcommand{\nsnonlinb}[2]{#1\cdot\grad #2}
\newcommand{\nsnonlin}[1]{\nsnonlinb{#1}{#1}}
\newcommand{\vvel}{v}
\newcommand{\vpres}{p}

\newcommand{\wvel}{w}

\newcommand{\Wvel}{W}

\newcommand{\uvel}{u}

\newcommand{\upres}{\mathfrak{p}}

\newcommand{\weakuvel}{\mathcal{U}}

\newcommand{\Oseensolopr}{S_\rey}
\newcommand{\Oseensoloprcompl}{S^\bot_\rey}

\newcommand{\fundsolvelkjl}{{\varGamma^\bot_{k,jl}}}
\newcommand{\fundsolvelkhh}{{\varGamma^\bot_{k,hh}}}
\newcommand{\fundsolvelcompl}{{\varGamma^{\rey}_\bot}}

\newcommand{\fundsoloseen}{{\varGamma^\rey_{{\mathrm{O}}}}}
\newcommand{\fundsolROseen}{{\gamma^{\rey}_{k}}}
\newcommand{\fundsolLaplace}{\varGamma_{\mathrm{L}}}

\newcommand{\tin}{\text{in }}
\newcommand{\tif}{\text{if }}

\newcommand{\tor}{\text{or }}
\newcommand{\tand}{\text{and }}

\newcommand{\half}{\frac{1}{2}}

\renewcommand{\epsilon}{\varepsilon}
\renewcommand{\phi}{\varphi}
\newcommand{\rey}{\lambda}
\newcommand{\halfrey}{\frac{\rey}{2}}
\newcommand{\tay}{\calt}
\newcommand{\per}{\tay}
\newcommand{\iper}{\frac{1}{\tay}}
\newcommand{\perf}{\frac{2\pi}{\tay}}

\newcommand{\eone}{\e_1}

\newcommand{\bigo}{O}

\newcommand{\Mmultiplier}{M}

\newcommand{\restterm}{\mathscr{R}}

\newcommand{\kroneckerdelta}{\delta}

\newcommand{\cutoff}{\chi}

\newcommand{\bijection}{\Pi}
\newcommand{\bijectioninv}{\Pi^{-1}}

\newcommand{\uf}{u}
\newcommand{\tI}{\tilde{I}}
\newcommand{\newCCtr}[2][d]{
\newcounter{#2}\setcounter{#2}{0}
\expandafter\xdef\csname kyedtheconst#2\endcsname{#1}
}
\newcommand{\Cc}[2][nolabel]{
\stepcounter{#2}
\expandafter\ensuremath{\csname kyedtheconst#2\endcsname_{\arabic{#2}}}
\ifthenelse{\equal{#1}{nolabel}}
{}
{\expandafter\xdef\csname kyedconst#1\endcsname
{\expandafter\ensuremath{\csname kyedtheconst#2\endcsname_{\arabic{#2}}}}}
}
\newcommand{\CcSetCtr}[2]{
\setcounter{#1}{#2}
}
\newcommand{\Cclast}[1]{
\expandafter\ensuremath{\csname kyedtheconst#1\endcsname_{\arabic{#1}}}
}
\newcommand{\Ccllast}[1]{
\addtocounter{#1}{-1}
\expandafter\ensuremath{\csname kyedtheconst#1\endcsname_{\arabic{#1}}}
\addtocounter{#1}{1}
}
\newcommand{\const}[1]{
\expandafter{\ifcsname kyedconst#1\endcsname
  \csname kyedconst#1\endcsname
\else
  \errmessage{Undefined Kyedconstant #1.}%
\fi}
}

\theoremstyle{plain}
\newtheorem{thm}{Theorem}[section]
\newtheorem{defn}[thm]{Definition}
\newtheorem{lem}[thm]{Lemma}

\newtheorem{cor}[thm]{Corollary}
\theoremstyle{remark}

\begin{document}
\title{Time-periodic solutions to the \\ Navier-Stokes equations in the three-dimensional whole-space with a non-zero drift term: Asymptotic profile at spatial infinity}

\author{
Giovanni P. Galdi\thanks{Partially supported by NSF-DMS grant 1614011.}\\ 
Department of Mechanical Engineering and Materials Science\\
University of Pittsburgh\\
Pittsburgh, PA 15261, USA\\
Email: \texttt{galdi@pitt.edu}
\and
Mads Kyed\\ 
Fachbereich Mathematik\\
Technische Universit\"at Darmstadt\\
Schlossgartenstr. 7, 64289 Darmstadt, Germany\\
Email: \texttt{kyed@mathematik.tu-darmstadt.de}\\
}

\date{\today}
\maketitle

\begin{abstract}
An asymptotic expansion at spatial infinity of a weak time-periodic solution to the Navier-Stokes equations with a non-zero drift term in the three-dimensional whole-space is carried out. The asymptotic profile is explicitly identified and expressed in terms of the well-known Oseen fundamental solution. A pointwise estimate is given
for the remainder term.
\end{abstract}


\newCCtr[C]{C}
\newCCtr[c]{c}
\let\oldproof\proof
\def\proof{\CcSetCtr{c}{-1}\oldproof} 
\newCCtr[M]{M}
\newCCtr[B]{B}
\newCCtr[\epsilon]{eps}
\CcSetCtr{eps}{-1}

\section{Introduction}

We investigate the asymptotic structure at spatial infinity of a time-periodic solution to a three-dimensional whole-space Navier-Stokes problem.
More specifically, we consider a solution $\uvel:\R\times\R^3\ra\R^3$ to the Navier-Stokes equations
\begin{align}\label{intro_nspastbodywholespace}
\begin{pdeq}
&\partial_t\uvel -\Delta\uvel -\rey\partial_{x_1}\uvel + \grad\upres + \nsnonlin{\uvel}= f && \tin\RthreeR,\\
&\Div\uvel =0 && \tin\RthreeR,\\
&\lim_{\snorm{x}\ra\infty}\uvel(t,x)=0
\end{pdeq}
\end{align} 
that is time-periodic with period $\per>0$,
\begin{align}\label{intro_timeperiodic}
\begin{aligned}
&\forall (t,x)\in\R\times\R^3:\quad \uvel(t,x) = \uvel(t+\per,x),
\end{aligned}
\end{align}
and corresponds to time-periodic data $f$ of the same period.  
In this context, $\RthreeR$ is a time-space domain. We denote by $t\in\R$ the time variable and by $x\in\R^3$ the spatial variable.
The velocity field $\uvel:\R\times\R^3\ra\R^3$ and scalar function $\upres:\R\times\R^3\ra\R$ represent the Eulerian velocity field 
and pressure term, respectively, of a fluid flow described by the Navier-Stokes equations \eqref{intro_nspastbodywholespace}.
Physically, \eqref{intro_nspastbodywholespace} models the flow of an
incompressible, viscous, Newtonian fluid past an object that moves with velocity $\rey\eone\in\R^3$. We shall consider the case $\rey\neq 0$ corresponding to the case
of an object moving with non-zero velocity. 
The motion of the object is then manifested in the so-called drift term $\rey\partial_{x_1}\uvel$ in \eqref{intro_nspastbodywholespace}.

The goal in the following is to establish an asymptotic expansion of a time-periodic solution $\uvel(t,x)$ as $\snorm{x}\ra\infty$, that is, a pointwise identity 
\begin{align}\label{intro_generalasymptoticexpansion}
\uvel(t,x) = \calg(t,x) + \restterm(t,x)
\end{align}
with $\calg$ an explicitly known vector field, which depends on the data $f$, and
$\restterm(t,x)$ a remainder term that decays faster than $\calg(t,x)$ as $\snorm{x}\ra\infty$. 
In this case, \eqref{intro_generalasymptoticexpansion} is an asymptotic expansion of $\uvel$ at spatial infinity with \emph{asymptotic profile} $\calg(t,x)$. 

In the investigation of Navier-Stokes problems in unbounded domains, information on the 
asymptotic structure of a solution at spatial infinity is an imperative for a comprehensive understanding of the corresponding fluid flow.
A classical result due to \textsc{Finn} \cite{Finn1959b,Finn1965a}, \textsc{Babenko} \cite{babenko1973} and \textsc{Galdi} \cite{Galdi1992b} states that 
an asymptotic expansion of a weak solution $\vvel:\R^3\ra\R^3$ to the three-dimensional \emph{steady-state} Navier-Stokes equations with a non-zero drift term is given by \begin{align}\label{intro_steadystateasympexp}
\vvel(x) = \fundsoloseen(x)\cdot\Bp{\int_{\R^3} f(x)\,\dx} + \bigo\bp{\snorm{x}^{-\frac{3}{2}+\epsilon}}\quad \text{for all }\epsilon>0,
\end{align}
where $\fundsoloseen$ denotes the well-known Oseen fundamental solution. Observe that a steady-state solution is, trivially, also time-periodic. In other words,
a time-independent solution to \eqref{intro_nspastbodywholespace} satisfies \eqref{intro_steadystateasympexp}.
In the following, we shall establish for any weak time-periodic solution to \eqref{intro_nspastbodywholespace} 
the asymptotic expansion
\begin{align}\label{intro_tpasympexp}
\uvel(t,x) = \fundsoloseen(x)\cdot\Bp{\iper\int_0^\per\int_{\R^3} f(t,x)\,\dx\dt} + \bigo\bp{\snorm{x}^{-\frac{3}{2}+\epsilon}}\quad \text{for all }\epsilon>0,
\end{align}
provided the solution possesses a certain amount of local integrability. Clearly, \eqref{intro_tpasympexp} is an extension of \eqref{intro_steadystateasympexp}
by which virtually all the physical properties 
that can be established for steady states from \eqref{intro_steadystateasympexp} can now also be established for time-periodic fluid flows. For example,
it is well known that \eqref{intro_steadystateasympexp} implies the existence of a wake region in the fluid flow, a property that 
renders the flow described by $\vvel$ reasonable from a physical point of view, since the Navier-Stokes equations with a drift terms describes the flow past an object. By \eqref{intro_tpasympexp} the same is true for time-periodic flows.

Existence of the type of time-periodic weak solutions that we shall consider in the following (Definition \ref{WeakSolClassDef}) was shown in \cite{habil} without any
restriction on the ``size'' of the data. Such a solution is therefore a natural starting point for further investigation. In order to establish \eqref{intro_tpasympexp},
however, we shall need to assume additional local integrability; see \eqref{MainThm_condonsol}. This condition is almost equivalent 
to a well-known integrability condition introduced for the corresponding initial-value problem in the pioneering works of \textsc{Leray}. Just as in the case 
of the corresponding initial-value problem, the condition implies enhanced regularity of the weak solution. 

Pointwise information such as \eqref{intro_tpasympexp} is typically derived from a fundamental solution to an appropriate linearization of the 
system of partial differential equations to which $\uvel$ is a solution. If we in \eqref{intro_nspastbodywholespace} neglect the 
nonlinear term, we obtain the time-periodic Oseen system. In the following, we shall identify what can be viewed as a fundamental solution to the time-periodic 
Oseen system. For this purpose, we formulate \eqref{intro_nspastbodywholespace} as a system of partial differential equations on the LCA group 
$\grp:=\R/\per\Z\times\R^3$ and utilize the Fourier transform $\FT_\grp$. We thereby obtain a fundamental solution for which pointwise 
estimates and integrability properties can be established. The proof of the main theorem is mainly based on these estimates and properties.

\section{Statement of the main result}\label{StatementOfMainResultSection}

The main result is an asymptotic expansion at spatial infinity of a weak solution $\uvel$ to \eqref{intro_nspastbodywholespace} 
with an explicit identification of the asymptotic profile and a pointwise estimate of the remainder term. 
We shall assume slightly more integrability of the solution than a weak solution possesses at the outset, but 
the expansion is obtained without any restriction on the ``size'' of the data $f$.

In order to state the main theorem, we introduce some function spaces. 
We first recall the notation $\CRcisigma(\R^3)$ for the function space of smooth solenoidal vector fields of compact 
support, and the notation 
\begin{align*}
\DSRNsigma{1}{2}(\R^3):=\overline{\CRcisigma(\R^3)}^{\norm{\grad\cdot}_{2}}=\setc{\uvel\in\LR{6}(\R^3)^3}{\grad\uvel\in\LR{2}(\R^3)^{3\times 3},\ \Div\uvel=0}
\end{align*}
for the homogeneous Sobolev space of solenoidal vector fields with finite Dirichlet integral (the latter equality above is due to the standard Sobolev embedding theorem). Moreover, we introduce the following spaces
\begin{align*}
&\CRper(\R\times\R^3) := \setcl{w\in\CRi(\R\times\R^3)}{\forall t\in\R:\ w(t+\per,\cdot)=w(t,\cdot)},\\
&\CRciper\bp{\R\times\R^3} := \setcl{w\in\CRper(\R\times\R^3)}{w\in\CRci([0,\per]\times\R^3)},\\
&\CRcisigmaper\bp{\R\times\R^3} := \setcl{w\in\CRciper(\R\times\R^3)^3}{\Div_x w = 0}
\end{align*}
of smooth $\per$-time-periodic functions. 
In addition, we define for sufficiently regular $\per$-time-periodic functions $u:\R\times\R^3\ra\R$
the operators
\begin{align}\label{intro_defofprojGernericExpression}
\proj u(t,x):=\iper\int_0^\per u(s,x)\,\ds\quad\tand\quad\projcompl u(t,x) := u(t,x)-\proj u(t,x).
\end{align}
Note that $\proj$ and $\projcompl$ decompose $u=\proj u + \projcompl u$ into a time-independent part
$\proj u$ and a time-periodic part $\projcompl u$ with vanishing time-average over the period.
Also note that $\proj$ and $\projcompl$ are complementary projections, that is, $\proj^2=\proj$ and $\projcompl=\id-\proj$. 
We shall sometimes refer to $\proj u$ as the \emph{steady state} part of $u$, and to $\projcompl u$ as the \emph{purely oscillatory} part of $u$.

An asymptotic expansion at spatial infinity will be established for a class of time-periodic weak solutions to \eqref{intro_nspastbodywholespace} that we call 
\emph{physically reasonable}. 

\begin{defn}\label{WeakSolClassDef}
Let $f\in\LRloc{1}\bp{\R\times\R^3}^3$ satisfy \eqref{intro_timeperiodic}.
We say that $\weakuvel\in\LRloc{1}\bp{\R\times\R^3}^3$ satisfying \eqref{intro_timeperiodic} is a \emph{physically reasonable weak time-periodic solution} to \eqref{intro_nspastbodywholespace} 
if
\begin{enumerate}[1),leftmargin=\parindent, itemindent=0.2cm]
\item\label{WeakSolClassDefProp1} $\weakuvel\in\LR{2}\bp{(0,\per);\DSRNsigma{1}{2}(\R^3)}$,
\item\label{WeakSolClassDefProp2} $\projcompl\weakuvel\in\LR{\infty}\bp{(0,\per);\LR{2}(\R^3)^3}$,
\item\label{WeakSolClassDefProp3} $\weakuvel$ is a generalized $\per$-time-periodic solution to \eqref{intro_nspastbodywholespace} in the sense that for all test functions $\Phi\in\CRcisigmaper\bp{\R\times\R^3}$ holds
\begin{align}\label{WeakSolClassDefDefofweaksol}
\begin{aligned}
\int_0^\per\int_{\R^3} -\weakuvel\cdot\partial_t\Phi +\grad\weakuvel:\grad\Phi -\rey\partial_1\weakuvel\cdot\Phi + (\nsnonlin{\weakuvel})\cdot\Phi\,\dx\dt  = \int_0^\per\int_{\R^3} f\cdot\Phi\,\dx\dt.
\end{aligned}
\end{align}  
\end{enumerate}
\end{defn}
The above class of physically reasonable weak solutions was introduced in \cite{habil}, where existence of a such a solution was shown for any 
$f\in\LR{2}\bp{(0,\per);\DSRN{-1}{2}(\R^3)^3}$ regardless of the data's ``size``. The class is therefore a natural starting point for further 
investigations. The description of the solutions as \emph{physically reasonable} is due to property \ref{WeakSolClassDefProp2}, which in physical
terms expresses that the kinetic energy of the purely oscillatory part $\projcompl\weakuvel$  is finite. In \cite{habil} it was further
required of a physically weak solution that it satisfies an energy inequality. We shall not need this property in the following and thus extend the class 
of physically weak solutions by leaving this condition out in Definition \ref{WeakSolClassDef} above.

The asymptotic profile in the expansion will be given in terms of the classical Oseen fundamental solution
\begin{align}\label{ae_defofoseenfundsol}
\begin{aligned}
&\fundsoloseen:\R^3\setminus\set{0}\ra\R^{3\times3},\quad\big[{\fundsoloseen}(x)\big]_{ij} := (\delta_{ij}\Delta-\partial_i\partial_j)\Phi^\rey(x),\\
&\Phi^\rey(x):=\frac{1}{4\pi\rey} \int_0^{\rey(\snorm{x}+x_1)/2} \frac{1-\e^{-\tau}}{\tau}\,\dtau.
\end{aligned}
\end{align} 
See also \cite[Chapter VII.3]{galdi:book1} for a closed-form expression of $\fundsoloseen$.

We are now in a position to state the main theorem of the paper.
\begin{thm}\label{MainThm}
Let $\rey\neq 0$ and $f\in\CRciper\bp{\R\times\R^3}^3$ be a $\per$-time-periodic vector-field.
If $\uvel$ is a physically reasonable weak time-periodic solution to \eqref{intro_nspastbodywholespace}, in the sense of 
Definition \ref{WeakSolClassDef}, which satisfies  
\begin{align}\label{MainThm_condonsol}
\exists r\in(5,\infty):\quad \projcompl\uvel\in\LR{r}\bp{(0,\per)\times\R^3}^3
\end{align}
then
\begin{align}\label{MainThm_asympexpansion}
\forall(t,x)\in\R\times\R^3:\quad \uvel(t,x) = \fundsoloseen(x)\cdot\Bp{\iper\int_0^\per\int_{\R^3} f\,\dx\dt} + \restterm(t,x)
\end{align}
with $\restterm$  satisfying
\begin{align}\label{MainThm_estofrestterm}
\forall\epsilon>0\ \exists\,\Cc[MainThm_estofresttermconst]{C}>0\ \forall\,\snorm{x}\geq 1,\ t\in\R:\quad \snorm{\restterm(t,x)} \leq \const{MainThm_estofresttermconst}\,\snorm{x}^{-\frac{3}{2}+\epsilon},
\end{align}
where $\const{MainThm_estofresttermconst}=\const{MainThm_estofresttermconst}(\epsilon)$.
\end{thm}
It is well-known that the Oseen fundamental solution $\fundsoloseen$ 
has order of decay $\snorm{x}^{-1+\sigma}$ as $\snorm{x}\ra\infty$ in a parabolic so-called wake region in the direction $\eone$, where
$\sigma$ depends on the angle of the specific parabolic region under consideration; see for example \cite[Chapter VII.3]{galdi:book1}. 
More precisely, $\sigma$ is zero if the angle is zero, and increases as the angle increases. 
From the asymptotic expansion established in \eqref{MainThm} one can therefore conclude the existence of such a 
wake region also for a physically reasonable weak time-periodic solution to Navier-Stokes system satisfying \eqref{MainThm_condonsol}
in the case $\rey\neq 0$. Note that the size of the manifested wake 
region is determined by the decay rate obtained for the remainder term $\restterm(t,x)$. 

The additional integrability  \eqref{MainThm_condonsol} is imposed for technical reasons only. It is needed to ensure additional local regularity of the weak solution. We emphasize that condition \eqref{MainThm_condonsol} does \emph{not} translate into 
additional decay of the solution at spatial infinity. The assumption can be removed if one is able to show sufficient local regularity for the weak solution;
an undertaking which seems just as challenging as solving the similar open problem for the initial-value Navier-Stokes problem though.
If the data $f$ is sufficiently ''small``, however, the additional regularity can be established; see \cite[Theorem 2.3]{ertpns}.

\section{Notation}

Points in $\R\times\R^3$ are denoted by $(t,x)$ with $t\in\R$ and $x\in\R^3$.
We refer to $t$ as the time variable and to $x$ as the spatial variable. 

We use $\B_R:=\setc{x\in\R^3}{\snorm{x}<R}$ to denote balls in $\R^3$. Furthermore, for $R_1<R_2$ we let 
$\B_{R_2,R_1}:=\setc{x\in\R^3}{R_1<\snorm{x}<R_2}$. In addition, we put $\B^R:=\R^3\setminus\B_R$. 

For a sufficiently regular function $u:\R\times\R^3\ra\R$, we put $\partial_i u:=\partial_{x_i} u$.

For two vectors $a,b\in\R^3$, we let $a \otimes b\in\R^{3\times 3}$ denote the tensor 
with $(a\otimes b)_{ij}:=a_ib_j$. 
We denote by $\idmatrix$ the identity tensor 
$\idmatrix\in\R^{3\times 3}$.

Constants in capital letters in the proofs and theorems are global, while constants in small letters are local to the proof in which they appear.
Unless otherwise stated, constants are positive.

We make use of Einstein's summation convention and
implicitly sum over all repeated indices appearing in a formula.

\section{Reformulation in a group setting}

We shall take advantage of the formalism developed in \cite{habil,mrtpns} to reformulate the time-periodic Navier-Stokes problem as a 
system of partial differential equation on the group
$\grp:=\R/\per\Z\times\R^3$.
We endow $\grp$ with the canonical topology and differentiable structure inherited from $\R\times\R^3$ via the quotient mapping
\begin{align*}
\quotientmap :\R\times\R^3 \ra \R/\per\Z\times\R^3,\quad \quotientmap(t,x):=([t],x).
\end{align*}
Clearly, $\grp$ is then a locally compact Abelian group.
There is a natural correspondence between $\per$-time-periodic functions defined on $\R\times\R^3$ and functions defined on $\grp$.
We shall take advantage of this correspondence and reformulate
\eqref{intro_nspastbodywholespace}--\eqref{intro_timeperiodic} and the main theorem in a setting of vector fields defined on $\grp$. 
The main advantages of this setting is the ability via the Fourier transform on $\grp$ 
to express solutions to systems of linear partial differential equations in terms of Fourier multipliers. 

The Haar measure $\dg$ on $\grp$ is unique up-to a normalization factor.  
Using the restriction $\bijection:=\pi_{|[0,\per)\times\R^3}$, which is clearly a (continuous) bijection between  $[0,\per)\times\R^3$ and $\grp$,
we choose the normalization factor so that
\begin{align*}
\int_\grp \uvel(g)\,\dg = \iper\int_0^\per\int_{\R^3} \uvel\circ\bijection(t,x)\,\dx\dt.
\end{align*}
For the sake of convenience, we will omit the symbol $\bijection$ in integrals with respect to $\dx\dt$ of $\grp$-defined functions, that is, instead of 
$\iper\int_0^\per\int_{\R^3} \uvel\circ\bijection(t,x)\,\dx\dt$ we simply write $\iper\int_0^\per\int_{\R^3} \uvel(t,x)\,\dx\dt$.

We shall make use of the Schwartz-Bruhat space of generalized Schwartz functions $\SR(\grp)$ 
and the corresponding dual space $\TDR(\grp)$ of tempered distributions; see for example \cite{Bruhat61,habil,mrtpns} for the exact definition. 
We identify the dual group $\dualgrp$ with $\Z\times\R^3$ by associating elements $(k,\xi)\in\Z\times\R^3$ with the characters 
$(t,x)\ra\e^{ix\cdot\xi+ik\perf t}$. In the following, we repeatedly use $(k,\xi)$ to denote points in $\dualgrp$.
The Fourier transform on $\grp$ is then given by the expression
\begin{align*}
\FT_\grp:\SR(\grp)\ra\SR(\dualgrp),\quad \FT_\grp(\uf)(k,\xi):=
\iper\int_0^\per\int_{\R^3} \uf(t,x)\,\e^{-ix\cdot\xi-i\perf k t}\,\dx\dt.
\end{align*}
Recall that 
$\FT_\grp:\SR(\grp)\ra\SR(\dualgrp)$ is a homeomorphism and extends by duality to a homeomorphism $\FT_\grp:\TDR(\grp)\ra\TDR(\dualgrp)$.

The space of smooth functions on $\grp$ is given by
\begin{align}\label{lt_smoothfunctionsongrp}
\CRi(\grp):=\setc{\uf:\grp\ra\R}{\uf\circ\quotientmap \in\CRi(\R\times\R^3)}.
\end{align}
For $\uf\in\CRi(\grp)$, derivatives are defined by 
\begin{align}\label{lt_defofgrpderivatives}
\forall(\alpha,\beta)\in\N_0^n\times\N_0:\quad \partial_t^\beta\partial_x^\alpha\uf := \bb{\partial_t^\beta\partial_x^\alpha (\uf\circ\quotientmap)}\circ\bijectioninv.
\end{align}
By $\CRci(\grp)$ we denote the subspace of $\CRi(\grp)$ of smooth functions with compact support. 
We further introduce for $r,s\in\N_0$ the anisotropic Sobolev space
\begin{align}\label{DefOfAnisotropicSobSpaceOnGrp}
\begin{aligned}
&\WSR{r,s}{q}\np{\grp} := \closure{\CRci\np{\grp}}{\norm{\cdot}_{r,s,q}},\\
&\norm{\uvel}_{r,s,q} := 
\Bp{\sum_{(\alpha,\beta)\in\N_0\times\N_0^3,\ \snorm{\alpha}\leq r,\snorm{\beta}\leq s} \norm{\partial_t^\alpha u}^q_{q} +\norm{\partial_x^\beta u}^q_{q}  
}^{1/q}. 
\end{aligned}
\end{align}
It is standard to verify that $\WSR{r,s}{q}\np{\grp}=\setc{\uf\in\LR{q}\np{\grp}}{\norm{\uf}_{r,s,q}<\infty}$.

The projections $\proj$ and $\projcompl$ are defined on $\grp$-defined functions by the same expressions as in \eqref{intro_defofprojGernericExpression}. 
Indeed, the integral in \eqref{intro_defofprojGernericExpression} is well-defined for any $\uvel\in\SR(\grp)$ and clearly both $\proj$ and $\projcompl$
map $\SR(\grp)$ into itself. By duality, both projections extend to continuous linear mappings 
$\proj:\TDR(\grp)\ra\TDR(\grp)$ and $\projcompl:\TDR(\grp)\ra\TDR(\grp)$. A direct computation shows that
\begin{align}\label{Reformulation_Symbolofproj}
\proj f = \iFT_\grp\bb{\projsymbol(k)\, \FT_\grp\nb{f}},\quad
\projcompl f = \iFT_\grp\bb{\bp{1-\projsymbol(k)}\, \FT_\grp\nb{f}},
\end{align}
where $\projsymbol$ denotes the delta distribution on $\Z$, that is,
\begin{align*}
\projsymbol:\Z\ra\CNumbers,\quad
\projsymbol(k):=
\begin{pdeq}
&1 && \tif k=0,\\
&0 && \tif k\neq0.
\end{pdeq}
\end{align*}
In other words, $\projsymbol$ and $(1-\projsymbol)$ is the Fourier symbol of the projection $\proj$ and $\projcompl$, respectively.

The classical Helmholtz-Weyl projection can also be extended from the Euclidean $\R^3$ setting to $\grp$-defined vector fields. It is convenient to do so by introducing
the Helmholtz-Weyl projection in terms of its Fourier symbol:
\begin{align}\label{Reformulation_HelmholtzProjDefDef}
\hproj: \LR{2}(\grp)^3\ra\LR{2}(\grp)^3,\quad \hproj f := \iFT_\grp\Bb{\Bp{\idmatrix - \frac{\xi\otimes\xi}{\snorm{\xi}^2}} \FT_\grp\nb{f}}.
\end{align}
It follows as in the Euclidean setting that $\hproj$ extends uniquely to a continuous projection $\hproj:\LR{q}(\grp)^3\ra\LR{q}(\grp)^3$ for any $q\in(1,\infty)$.

We now consider for data $F$ the linear problem 
\begin{align}\label{Reformulation_eqforWvel}
\begin{pdeq}
&\partial_t\Wvel -\Delta\Wvel -\rey\partial_1\Wvel  = \projcompl\hproj F && \tin\grp,\\
&\Div\Wvel =0 && \tin\grp.
\end{pdeq}
\end{align} 
If we consider the system in the realm of tempered distributions $\TDR(\grp)$ and apply the Fourier transform $\FT_\grp$, we find, formally at least, the expression
\begin{align}\label{Reformulation_RepFormulaForWvel}
\Wvel = \iFT_\grp\Bb{ \frac{1-\projsymbol(k)}{\snorm{\xi}^2 +i(\perf k-\rey \xi_1) }\,\FT_{\grp}\bb{\hproj F}}
\end{align}
for the solution. 
We shall briefly recall some result from \cite{habil,mrtpns} concerning the validity of this representation formula. For this purpose, we introduce
the linear operators
\begin{align*}
&\Oseensoloprcompl:\ \SR(\grp)^3 \ra \TDR(\grp)^3,\  {\Oseensoloprcompl\psi} := 
\iFT_\grp\Bb{ \frac{1-\projsymbol(k)}{\snorm{\xi}^2 +i(\perf k-\rey \xi_1) }\,\FT_{\grp}\bb{\psi}},\\
&\Oseensoloprcompl\circ\Div:\ \SR(\grp)^{3\times 3} \ra \TDR(\grp)^3,\ \bp{\Oseensoloprcompl\circ\Div\psi}_h := 
\iFT_\grp\Bb{\frac{\np{1-\projsymbol(k)}\,i\xi_j}{\snorm{\xi}^2 +i(\perf k-\rey \xi_1) }{\FT_{\grp}\bb{\psi_{jh}}}}.
\end{align*}
It is easy to see that $\Oseensoloprcompl$ and $\Oseensoloprcompl\circ\Div$ are well-defined in the setting of tempered distributions above. Their mapping properties in an $\LR{q}(\grp)$ setting are given in the following lemma. 
These properties hold for all $\rey\in\R$, so for the sake of completeness we make no restriction to $\rey\neq 0$ in the follow lemmas.

\begin{lem}\label{Reformulation_OseensoloprcomplMappingProps}
Let $\rey\in\R$ and $q\in(1,\infty)$. The operator $\Oseensoloprcompl$ extends uniquely to a bounded linear operator
\begin{align}
&\Oseensoloprcompl:\ \LR{q}(\grp)^3\ra\WSR{1,2}{q}(\grp)^3.\label{Reformulation_OseensoloprcomplMappingPropsProp1}
\end{align}
Let $r:=\frac{5q}{5-q}$ if $q\in(1,5)$, let $r\in(5,\infty)$ if $q=5$, and let $r:=\infty$ if $q\in(5,\infty)$.
Then the operator $\Oseensoloprcompl\circ \Div$ extends uniquely to a bounded linear operator
\begin{align}
&\Oseensoloprcompl\circ \Div:\ \LR{q}(\grp)^{3 \times 3}\ra \WSR{0,1}{q}(\grp)^3\cap\LR{r}(\grp)^3.\label{Reformulation_OseensoloprcomplMappingPropsProp2}
\end{align}
\end{lem}

\begin{proof}
Property \eqref{Reformulation_OseensoloprcomplMappingPropsProp1} was shown in \cite[Proof of Theorem 4.8]{mrtpns}.
The proof hereof relies on the fact that the multiplier
\begin{align*}
\Mmultiplier_1:\dualgrp\ra\CNumbers,\quad \Mmultiplier_1(k,\xi):=\frac{1-\projsymbol(k)}{\snorm{\xi}^2 +i(\perf k-\rey \xi_1) }
\end{align*}
that defines $\Oseensoloprcompl$ is smooth, that is, $\Mmultiplier_1\in\CRi(\dualgrp)$, and possesses sufficient decay as $\snorm{\xi}\ra\infty$ and $\snorm{k}\ra\infty$. 
Observe in particular that the numerator $1-\projsymbol(k)$ vanishes in a neighborhood of the only zero at $(0,0)$ of the denominator of $\Mmultiplier_1$.
With the technique from \cite[Proof of Theorem 4.8]{mrtpns}, also the multiplier 
\begin{align*}
\Mmultiplier_2:\dualgrp\ra\CNumbers,\quad \Mmultiplier_2(k,\xi):=\frac{\bp{1-\projsymbol(k)}\,i\xi_j}{\snorm{\xi}^2 +i(\perf k-\rey \xi_1) }
\end{align*}
that defines $\Oseensoloprcompl\circ \Div$ can be analyzed to show 
\begin{align}\label{Reformulation_OseensoloprcomplMappingPropsProp2Est1}
\norm{\grad\Oseensoloprcompl\circ \Div \psi}_q \leq \Cc{c}\,\norm{\psi }_q.
\end{align}
Finally, 
since $\FT_\grp=\FT_{\R/{2\pi\Z}}\circ\FT_{\R^3}$, it follows that
\begin{align*}
\Oseensoloprcompl\circ \Div \psi &= \iFT_{\R/{2\pi\Z}}\Bb{\bp{1-\projsymbol(k)}\snorm{k}^{-\frac{1}{5}}}*_{\R/{2\pi\Z}}\iFT_{\R^3}\Bb{\snorm{\xi}^{-\frac{3}{5}}}*_{\R^3}
\iFT_\grp\Bb{ \Mmultiplier_3\, \FT_\grp\bb{\psi}},
\end{align*}
with
\begin{align*}
\Mmultiplier_3:\dualgrp\ra\CNumbers,\quad \Mmultiplier_3(k,\xi):= \frac{\snorm{k}^{\frac{1}{5}}\,\snorm{\xi}^{\frac{3}{5}}\,\bp{1-\projsymbol(k)}\,i\xi_j}{\snorm{\xi}^2 +i(\perf k-\rey \xi_1) }
\end{align*}
Again with the technique from \cite[Proof of Theorem 4.8]{mrtpns} it can be shown that $\Mmultiplier_3$ is an $\LR{q}(\grp)$ multiplier. The estimate
\begin{align}\label{Reformulation_OseensoloprcomplMappingPropsProp2Est2}
\norm{\Oseensoloprcompl\circ \Div \psi}_r \leq \Cc{c}\,\norm{\psi }_q.
\end{align}
thus follows from embedding properties of the two Riesz potentials 
\begin{align*}
\psi\ra\iFT_{\R^3}\Bb{\snorm{\xi}^{-\frac{3}{5}}}*_{\R^3}\psi,\quad \psi\ra\ \iFT_{\R/{2\pi\Z}}\Bb{\bp{1-\projsymbol(k)}\snorm{k}^{-\frac{1}{5}}}*_{\R/{2\pi\Z}}.
\end{align*}
For the details we refer to \cite[Proof of Theorem 4.1]{tpfpb}. 
By \eqref{Reformulation_OseensoloprcomplMappingPropsProp2Est1} and \eqref{Reformulation_OseensoloprcomplMappingPropsProp2Est2} we conclude \eqref{Reformulation_OseensoloprcomplMappingPropsProp2}.
\end{proof}

\begin{lem}\label{Reformulation_UniquenessComplEq}
Let $\rey\in\R$. If $\uvel\in\TDR(\grp)$ with $\proj\uvel=0$ satisfies
$\partial_t\uvel -\Delta\uvel -\rey\partial_1\uvel = 0$,
then $\uvel=0$.
\end{lem}
\begin{proof}
Applying the Fourier transform $\FT_\grp$, we see that $\bp{\snorm{\xi}^2 +i(\perf k-\rey \xi_1)}\ft{\uvel}=0$. It follows that 
$\supp\ft{\uvel}\subset\set{(0,0)}$. However, since $\proj\uvel=0$ we have $(0,0)\notin \supp\ft{\uvel}$. Consequently $\supp\ft{\uvel}=\emptyset$ and thus $\uvel=0$.
\end{proof}

By Lemma \ref{Reformulation_OseensoloprcomplMappingProps} and Lemma \ref{Reformulation_UniquenessComplEq} we have obtained a statement concerning existence 
and uniqueness of a solution to \eqref{Reformulation_eqforWvel}. More specifically, for data $F\in\LR{q}(\grp)^3$ we find that $\Wvel:=\Oseensoloprcompl\hproj F$ 
is a solution to \eqref{Reformulation_eqforWvel} in $\WSR{1,2}{q}(\grp)^3$ that satisfies the representation formula \eqref{Reformulation_RepFormulaForWvel}. Moreover, this
solution is unique in the subspace $\projcompl\TDR(\grp)^3$ of tempered distributions.
From \eqref{Reformulation_RepFormulaForWvel} we observe, again formally, that $\Wvel=\fundsolvelcompl*F$ with 
\begin{align}\label{Reformulation_defoffundsol}
\fundsolvelcompl\in\TDR(\grp)^{3\times 3},\quad
\fundsolvelcompl := \iFT_\grp\Bb{\frac{1-\projsymbol(k)}{\snorm{\xi}^2 +i(\perf k -\rey \xi_1)}\Bp{\idmatrix-\frac{\xi\otimes\xi}{\snorm{\xi}^2}}}.
\end{align}
In the following, we shall establish both pointwise estimates and integrability properties of $\fundsolvelcompl$, and as a consequence we will
see that the solution $\Wvel$ can indeed be legitimately expressed as $\Wvel=\fundsolvelcompl*F$. Thus, 
we may view $\fundsolvelcompl$ as a fundamental solution to \eqref{Reformulation_eqforWvel}. The rest of this section is devoted to an analysis of $\fundsolvelcompl$.

\begin{lem}\label{ae_estofroseenfundsol}
Let $\rey\in\R$, $k\in\Z\setminus\set{0}$ and 
\begin{align}\label{ae_estofroseenfundsolDef}
\fundsolROseen:\R^3\setminus\set{0}\ra\CNumbers,\quad 
\fundsolROseen(x):=\frac{1}{4\pi\snorm{x}}\e^{-\bp{i\perf k+ (\halfrey)^2}^\half\snorm{x}-\halfrey x_1},
\end{align}
where $\bp{i\perf k+ (\halfrey)^2}^\half$ denotes the square root with non-negative real part.
Then 
\begin{align}
&\snorm{\fundsolROseen(x)}\leq \frac{1}{4\pi\snorm{x}}\e^{-\Cc[ae_estofroseenfundsolpestexponentconst]{C}\snorm{k}^\half\snorm{x}},\label{ae_estofroseenfundsolpest} \\
&\snorml{\partial_j\bb{\fundsolROseen(x)}}\leq \Cc[ae_estofroseenfundsolpestgradConst]{C} \Bp{ \frac{1}{\snorm{x}^2}+ 
\frac{\snorm{k}^\half}{\snorm{x}}}\e^{-\const{ae_estofroseenfundsolpestexponentconst}\snorm{k}^\half\snorm{x}},\label{ae_estofroseenfundsolpestgrad}\\
&\snorml{\partial_j\partial_h\bb{\fundsolROseen(x)}}\leq \Cc[ae_estofroseenfundsolpestgradgradConst]{C} \Bp{ \frac{1}{\snorm{x}^3}+ 
\frac{\snorm{k}^\half}{\snorm{x}^2}+\frac{\snorm{k}}{\snorm{x}}}\e^{-\const{ae_estofroseenfundsolpestexponentconst}\snorm{k}^\half\snorm{x}},\label{ae_estofroseenfundsolpestgradgrad}
\end{align}
with $\const{ae_estofroseenfundsolpestexponentconst}=\const{ae_estofroseenfundsolpestexponentconst}(\rey,\per)$, $\const{ae_estofroseenfundsolpestgradConst}=\const{ae_estofroseenfundsolpestgradConst}(\rey,\per)$
and $\const{ae_estofroseenfundsolpestgradgradConst}=\const{ae_estofroseenfundsolpestgradgradConst}(\rey,\per)$.
\end{lem}

\begin{proof}\newCCtr[c]{ae_estofroseenfundsol}
The case $\rey=0$ is trivial, so we assume $\rey\neq 0$.
We first observe that
\begin{align*}
&\realpart\bigg[ \big(i\perf k + (\rey/2)^2 \big)^\half\bigg] =
\labs{i\perf k+(\rey/2)^2}^{\half} \cos\bigg( \half \arctan\bigg(\frac{\perf k}{(\rey/2)^2}\bigg)\bigg) \\
&\qquad\qquad= \big((\perf)^2 k^2+(\rey/2)^4\big)^{\frac{1}{4}} \frac{1}{\sqrt{2}}\bigg(1+\cos\bigg( \arctan\bigg(\frac{\perf k}{(\rey/2)^2}\bigg)\bigg)\bigg)^\half\\
&\qquad\qquad= \big((\perf)^2 k^2+(\rey/2)^4\big)^{\frac{1}{4}} \frac{1}{\sqrt{2}}\bigg(1+
\bigg(1+\frac{(\perf)^2 k^2}{(\rey/2)^4}\bigg)^{-\half}\bigg)^\half\\
&\qquad\qquad= \frac{1}{\sqrt{2}} (\snorm{\rey}/2) \bigg( \bigg(1+\frac{(\perf)^2 k^2}{(\rey/2)^4}\bigg)^{\half}+1\bigg)^\half.
\end{align*} 
It follows that 
\begin{align*}
\forall k\in\Z\setminus{\set{0}}:\ \realpart\bigg[ \big(i\perf k + (\rey/2)^2 \big)^\half - (\rey/2)\bigg] > 0
\end{align*}
and 
\begin{align*}
\lim_{\snorm{k}\ra\infty}\frac{\realpart\big[ \big(i\perf k + (\rey/2)^2 \big)^\half - (\rey/2)\big]}{\snorm{k}^\half}
= \sqrt{\frac{\perf}{2}}.
\end{align*}
Consequently, there is a constant $\Cc[ae_estofroseenfundsolexpconstpre]{ae_estofroseenfundsol}=\const{ae_estofroseenfundsolexpconstpre}(\rey,\per)>0$ such that
\begin{align*}
\forall k\in\Z\setminus{\set{0}}:\ \realpart\bigg[\big(i\perf k + (\rey/2)^2 \big)^\half - (\rey/2)\bigg] \geq 
\const{ae_estofroseenfundsolexpconstpre}\snorm{k}^\half,
\end{align*}
from which we conclude that 
\begin{align*}
\begin{aligned}
&\realpart\bigg[ -\big(i\perf k + (\rey/2)^2 \big)^\half\snorm{x}-(\rey/2) x_1 \bigg]\\
&\qquad \leq -\realpart\bigg[\big(i\perf k + (\rey/2)^2 \big)^\half-(\rey/2) \bigg]\snorm{x} 
\leq -\const{ae_estofroseenfundsolexpconstpre} \snorm{k}^\half \snorm{x}.
\end{aligned}
\end{align*}
We can now estimate
\begin{align}\label{ae_estofroseenfundsolpestpest}
\begin{aligned}
\snorm{\fundsolROseen(x)}&\leq \frac{1}{4\pi\snorm{x}}\e^{\realpart\big[ -\big(i\perf k + (\rey/2)^2 \big)^\half\snorm{x}-(\rey/2) x_1 \big]}\\
&\leq \frac{1}{4\pi\snorm{x}}\e^{-\const{ae_estofroseenfundsolexpconstpre}\snorm{k}^\half\snorm{x}},
\end{aligned}
\end{align}
and conclude \eqref{ae_estofroseenfundsolpest}. Similarly, one may verify \eqref{ae_estofroseenfundsolpestgrad} and \eqref{ae_estofroseenfundsolpestgradgrad}
after taking derivatives.
\end{proof}

We can now establish a pointwise estimate of $\fundsolvelcompl$.

\begin{lem}\label{ae_estoffundsolcompl}
Let $\rey\in\R$ and $\fundsolvelcompl$ be defined as in \eqref{Reformulation_defoffundsol}. 
Then ($j,l,h=1,2,3$)
\begin{align}
&\forall \snorm{x}\geq 2:\quad \Bp{\iper\int_0^\per\snorml{\bb{\fundsolvelcompl}_{jl}(t,x)}^2\,\dt}^\half \leq \frac{\Cc[ae_estoffundsolcomplconst]{C}}{\snorm{x}^3},\label{ae_estoffundsolcomplest}\\
&\forall \snorm{x}\geq 2:\quad \Bp{\iper\int_0^\per\snorml{\partial_h\bb{\fundsolvelcompl}_{jl}(t,x)}^2\,\dt}^\half \leq \frac{\const{ae_estoffundsolcomplconst}}{\snorm{x}^4},\label{ae_estoffundsolcomplgradest}
\end{align}
with $\const{ae_estoffundsolcomplconst}=\const{ae_estoffundsolcomplconst}(\rey,\per)$.
\end{lem}

\begin{proof}\newCCtr[c]{ae_estoffundsolcompl}
For $k\in\Z\setminus\set{0}$ we put
\begin{align*}
\fundsolvelkjl:= \iFT_{\R^3}\Bb{\frac{\xi_j\xi_l}{\snorm{\xi}^2}\cdot \frac{1}{\snorm{\xi}^2 +i(\perf k -\rey \xi_1)} } \in\TDR(\R^3).
\end{align*}
A direct computation yields
\begin{align}\label{ae_estoffundsolcomplrepasconvolution}
\fundsolvelkjl = \partial_j\partial_l\iFT_{\R^3}\bb{\FT_{\R^3}\bp{\fundsolLaplace}\cdot \FT_{\R^3}\bp{\fundsolROseen}}
\end{align}
with
\begin{align*}
\fundsolLaplace:\R^3\setminus\set{0}\ra\CNumbers,\quad \fundsolLaplace(x):= \frac{1}{4\pi\snorm{x}}
\end{align*}
and $\fundsolROseen$ given by \eqref{ae_estofroseenfundsolDef}.
In view of \eqref{ae_estofroseenfundsolpest}, we see that both $\fundsolLaplace$ and $\fundsolROseen$ are ``regular'' enough
to express the right-hand side of \eqref{ae_estoffundsolcomplrepasconvolution} in terms of a classical convolution integral. 
More precisely, we have
\begin{align}\label{ae_estoffundsolcomplrepasclassicalconvolution}
\fundsolvelkjl(x) = \partial_j\partial_l \int_{\R^3}\fundsolLaplace(x-y)\,\fundsolROseen(y)\,\dy.
\end{align}
We introduce at this point a ``cut-off'' function $\cutoff\in\CRci(\R;\R)$ with
\begin{align*}
\cutoff(x)=
\begin{pdeq}
&0 && \tif 0\leq \snorm{x}\leq \half\ \ \tor\ 4\leq \snorm{x},\\
&1 && \tif 1\leq \snorm{x}\leq 3,
\end{pdeq}
\end{align*}
and define for $R>0$
\begin{align*}
\cutoff_R\in\CRci(\R^3;\R),\quad \cutoff_R(x):=\cutoff(\snorm{x}/R).
\end{align*}
Now consider an $x\in\R^3$ with $\snorm{x}>2$. Put $R:=\frac{\snorm{x}}{2}$. Note that this implies $x\in\B_{3R,R}$.
We use $\cutoff_R$ to decompose the convolution integral in \eqref{ae_estoffundsolcomplrepasclassicalconvolution} as follows:
\begin{align}\label{ae_estoffundsolcomplConvIntSplit}
\begin{aligned} 
\fundsolvelkjl(x)
&=\ \partial_j\partial_l \int_{\B_{4R,R/2}}\fundsolLaplace(x-y)\,\fundsolROseen(y)\,\cutoff_R(y)\, \dy\\
&\quad +\partial_j\partial_l \int_{\B^{3R}}\fundsolLaplace(x-y)\,\fundsolROseen(y)\,\bp{1-\cutoff_R(y)}\, \dy\\
&\quad+\partial_j\partial_l \int_{\B_{R}}\fundsolLaplace(x-y)\,\fundsolROseen(y)\,\bp{1-\cutoff_R(y)}\, \dy\\
&=:I_1(x)+I_2(x)+I_3(x).
\end{aligned}
\end{align}
We can estimate
\begin{align}\label{ae_estoffundsolcomplestofI1a}
\begin{aligned}
\snorm{I_1(x)}\leq \int_{\B_{4R,R/2}}\frac{1}{4\pi\snorm{x-y}^2}\,\snorml{\partial_j\bb{\fundsolROseen(y)\,\cutoff_R(y)}}\, \dy.
\end{aligned}
\end{align}
Utilizing \eqref{ae_estofroseenfundsolpest} and \eqref{ae_estofroseenfundsolpestgrad}, we see that for
$y\in\B_{4R,R/2}$ holds
\begin{align*}
\begin{aligned}
\snorml{\partial_j\bb{\fundsolROseen(y)\,\cutoff_R(y)}}\leq 
\Cc{ae_estoffundsolcompl}\frac{1}{\snorm{k}R^{4}},
\end{aligned}
\end{align*}
with $\Cclast{ae_estoffundsolcompl}=\Cclast{ae_estoffundsolcompl}(\rey,\per)$.
It follows that 
\begin{align}\label{ae_estoffundsolcomplestofI1final}
\begin{aligned}
\snorm{I_1(x)}\leq 
\Cc{ae_estoffundsolcompl}\frac{1}{\snorm{k}R^{3}},
\end{aligned}
\end{align}
with $\Cclast{ae_estoffundsolcompl}=\Cclast{ae_estoffundsolcompl}(\rey,\per)$.
Using again \eqref{ae_estofroseenfundsolpest}, we further estimate
\newlength{\kyedtmpsize}
\settoheight{\kyedtmpsize}{\mbox{$\snorm{k}^{\frac{3}{2}}$}}
\newlength{\kyedtmpsizeb}
\settoheight{\kyedtmpsizeb}{\mbox{$\snorm{k}^{3}$}}
\begin{align}\label{ae_estoffundsolcomplestofI2final}
\begin{aligned}
\snorm{I_2(x)} &\leq \Cc{ae_estoffundsolcompl}\int_{\B^{3R}}\frac{1}{\snorm{x-y}^3}\,\snorml{\fundsolROseen(y)}\, \dy
\leq \Cc{ae_estoffundsolcompl}\int_{\B^{3R}}\frac{1}{\snorm{x-y}^3}\,\frac{1}{\text{\raisebox{0pt}[\kyedtmpsizeb]{$\snorm{y}$}}}
\e^{-\const{ae_estofroseenfundsolpestexponentconst}\snorm{k}^\half\snorm{y}}\, \dy\\
&\leq \Cc{ae_estoffundsolcompl}\int_{\B^{3R}}\frac{1}{\text{\raisebox{0pt}[\kyedtmpsize]{$R^3$}}}\,\frac{1}{\text{\raisebox{0pt}[\kyedtmpsize]{$\snorm{y}$}}}
\frac{1}{\snorm{k}^{\frac{3}{2}}\snorm{y}^3}\, \dy
\leq \Cc{ae_estoffundsolcompl}\frac{1}{\snorm{k}R^{3}},
\end{aligned}
\end{align}
with $\Cclast{ae_estoffundsolcompl}=\Cclast{ae_estoffundsolcompl}(\rey,\per)$, 
and 
\begin{align}\label{ae_estoffundsolcomplestofI3final}
\begin{aligned}
\snorm{I_3(x)} &\leq \Cc{ae_estoffundsolcompl}\int_{\B_{R}}\frac{1}{\snorm{x-y}^3}\,\snorml{\fundsolROseen(y)}\, \dy
\leq \Cc{ae_estoffundsolcompl}\int_{\B_{R}}\frac{1}{R^3}\,\frac{1}{\snorm{y}}\e^{-\const{ae_estofroseenfundsolpestexponentconst}\snorm{k}^\half\snorm{y}}\, \dy\\
&= \Cc{ae_estoffundsolcompl}\frac{1}{R^3}\,\int_0^R
r\e^{-\const{ae_estofroseenfundsolpestexponentconst}\snorm{k}^\half r}\, \dr
\leq \Cclast{ae_estoffundsolcompl}\frac{1}{\snorm{k}R^3}\,\int_0^\infty s\e^{-\const{ae_estofroseenfundsolpestexponentconst}s}\, \ds
= \Cc{ae_estoffundsolcompl}\frac{1}{\snorm{k}R^3},
\end{aligned}
\end{align}
with $\Cclast{ae_estoffundsolcompl}=\Cclast{ae_estoffundsolcompl}(\rey,\per)$. Combining \eqref{ae_estoffundsolcomplestofI1final},
\eqref{ae_estoffundsolcomplestofI2final} and \eqref{ae_estoffundsolcomplestofI3final}, we deduce
\begin{align*}
\snorml{\fundsolvelkjl(x)}\leq\Cc{ae_estoffundsolcompl}\frac{1}{\snorm{k}\snorm{x}^3},
\end{align*}
with $\Cclast{ae_estoffundsolcompl}=\Cclast{ae_estoffundsolcompl}(\rey,\per)$.
Denoting by $\FT_{\R/\per\Z}$ the Fourier transform on the torus $\R/\per\Z$, we employ Plancherel's theorem to deduce
\begin{align*}
\begin{aligned}
&\Bp{\iper\int_0^\per\snorm{\bb{\fundsolvelcompl}_{jl}(t,x)}^2\,\dt}^\half\\
&\qquad= \Bp{\iper\int_0^\per\snormLL{\iFT_{\R/\per\Z}\Bb{\bp{1-\projsymbol(k)}\bp{\kroneckerdelta_{jl}\fundsolvelkhh(x)-\fundsolvelkjl(x)}}}^2\,\dt}^\half\\
&\qquad= \norm{{\bp{1-\projsymbol(k)}\bp{\kroneckerdelta_{jl}\fundsolvelkhh(x)-\fundsolvelkjl(x)}}}_{\lR{2}(\Z)}\\
&\qquad\leq \Cc{ae_estoffundsolcompl}\,\Bp{\sum_{k\in\Z\setminus\set{0}}\frac{1}{\bp{\snorm{k}\snorm{x}^3}^2}}^\half
\leq \Cc{ae_estoffundsolcompl}\,\frac{1}{\snorm{x}^3},
\end{aligned}
\end{align*}
which yields \eqref{ae_estoffundsolcomplest}. 
Differentiating both sides in \eqref{ae_estoffundsolcomplConvIntSplit} we obtain a formula for $\partial_h\fundsolvelkjl$
based on which we can proceed as above 
to deduce \eqref{ae_estoffundsolcomplgradest}.
\end{proof}

\begin{lem}\label{ae_integrabilityoffundsolcompl}
Let $\rey\in\R$ and $\fundsolvelcompl$ be defined as in \eqref{Reformulation_defoffundsol}. 
Then
\begin{align}
&\forall r\in\Bp{1,\frac{5}{3}}:\quad \fundsolvelcompl\in\LR{r}(\grp)^{3\times 3},\label{ae_integrabilityoffundsolcomplsummability}\\
&\forall r\in\Bbp{1,\frac{4}{3}}:\quad \partial_j\fundsolvelcompl\in\LR{r}(\grp)^{3\times 3}\quad (j=1,2,3).\label{ae_integrabilityoffGradundsolcomplsummability}
\end{align}
\end{lem}

\begin{proof}\newCCtr[c]{ae_integrabilityoffundsolcompl}
We recall the definition \eqref{Reformulation_defoffundsol} of $\fundsolvelcompl$ to find that
\begin{align*}
\bb{\fundsolvelcompl}_{jl} = \bb{\kroneckerdelta_{jl}(\riesztrans_{h}\riesztrans_{h})-\riesztrans_{j}\riesztrans_{l}}\circ
\iFT_\grp\Bb{\frac{\bp{1-\projsymbol(k)} \snorm{k}^{\frac{2}{5}} \bp{1+\snorm{\xi}^2}^{\frac{3}{5}} }{\snorm{\xi}^2 +i(\perf k -\rey \xi_1)}\FT_\grp\bb{\iFT_\grp\bp{\calk}}},
\end{align*}
where $\riesztrans_{j}$ denotes Riesz transform 
\begin{align*}
\riesztrans_{j}:\SR(\grp)\ra\TDR(\grp),\quad \riesztrans_{j}(f):=\iFT_{\R^3}\Bb{\frac{\xi_j}{\snorm{\xi}}\FT_{\R^3}\np{f}}
\end{align*}
and 
\begin{align*}
\calk:\Z\times\R^3\ra\CNumbers,\quad \calk(k,\xi):= \bp{1-\projsymbol(k)} \snorm{k}^{-\frac{2}{5}} \bp{1+\snorm{\xi}^2}^{-\frac{3}{5}}. 
\end{align*}
By well-known theory, $\riesztrans_{j}$ extends for all $q\in(1,\infty)$ to a bounded linear operator
$\riesztrans_{j}:\LR{q}(\grp)\ra\LR{q}(\grp)$; see for example \cite[Corollary 4.2.8]{Grafakos1}.
Moreover, one can verify with the technique from \cite[Proof of Theorem 4.8]{mrtpns} that  
\begin{align*}
\Mmultiplier:\dualgrp\ra\CNumbers,\quad \Mmultiplier(k,\xi):=\frac{\bp{1-\projsymbol(k)} \snorm{k}^{\frac{2}{5}} \bp{1+\snorm{\xi}^2}^{\frac{3}{5}} }{\snorm{\xi}^2 +i(\perf k -\rey \xi_1)}
\end{align*}
is an $\LR{q}(\grp)$ multiplier. Thus, \eqref{ae_integrabilityoffundsolcomplsummability} follows if we can show that $\iFT_\grp\np{\calk}\in\LR{r}(\grp)$. 
Choosing $[-\half\per,\half\per)$ as a realization of $\Torus$, it follows that 
\begin{align}\label{ae_integrabilityoffundsolcomplInvFT1}
\iFT_{\Torus}\Bb{\bp{1-\projsymbol(k)} \snorm{k}^{-\frac{2}{5}}}(t)=t^{-\frac{3}{5}}+ h(t),
\end{align}
where $h\in\CRi\np{\Torus}$; see for example \cite[Example 3.1.19]{Grafakos1}.
Additionally, one can estimate
\begin{align}\label{ae_integrabilityoffundsolcomplInvFT2}
\snormLL{\iFT_{\R^3}\Bb{\bp{1+\snorm{\xi}^2}^{-\frac{3}{5}}}(x)} \leq 
\Cc{ae_integrabilityoffundsolcompl} \Bp{\snorm{x}^{-\frac{9}{5}}\cutoff_{\B_2}(x)+ \e^{-\frac{\snorm{x}}{2}}}; 
\end{align}
see for example \cite[Proposition 6.1.5]{Grafakos2}. We can therefore deduce
\begin{align*}
\iFT_\grp\np{\calk} 
&= \iFT_{\Torus}\Bb{\bp{1-\projsymbol(k)} \snorm{k}^{-\frac{2}{5}}}\otimes \iFT_{\R^3}\Bb{\bp{1+\snorm{\xi}^2}^{-\frac{3}{5}}}\in\LR{r}(\grp)
\end{align*}
for all $r\in\bp{1,\frac{5}{3}}$, which concludes \eqref{ae_integrabilityoffundsolcomplsummability}. To show \eqref{ae_integrabilityoffGradundsolcomplsummability},
we observe that 
\begin{align*}
\partial_j\bb{\fundsolvelcompl}_{jl} = \bb{\kroneckerdelta_{jl}(\riesztrans_{h}\riesztrans_{h})-\riesztrans_{j}\riesztrans_{l}}\circ
\iFT_\grp\Bb{\frac{\bp{1-\projsymbol(k)} \snorm{k}^{\frac{1}{4}} \bp{1+\snorm{\xi}^2}^{\frac{3}{8}} \xi_j }{\snorm{\xi}^2 +i(\perf k -\rey \xi_1)}\FT_\grp\bb{\iFT_\grp\bp{\calj}}},
\end{align*}
where 
\begin{align*}
\calj:\Z\times\R^3\ra\CNumbers,\quad \calj(k,\xi):= \bp{1-\projsymbol(k)} \snorm{k}^{-\frac{1}{4}} \bp{1+\snorm{\xi}^2}^{-\frac{3}{8}}. 
\end{align*}
We then obtain \eqref{ae_integrabilityoffGradundsolcomplsummability} for all $q\in(1,\frac{4}{3})$ by the same computations as above.
It follows in particular that $\partial_j\fundsolvelcompl\in\LRloc{1}\np{\grp}$. The asymptotic decay as $\snorm{x}\ra\infty$ established in 
\eqref{ae_estoffundsolcomplgradest} in Lemma \ref{ae_estoffundsolcompl} therefore allows us to further conclude $\partial_j\fundsolvelcompl\in\LR{1}\np{\grp}$.
\end{proof}

\section{Proof of the main theorem}

We shall make use of the framework for the linearized time-periodic Navier-Stokes system developed in 
\cite{mrtpns} and \cite{habil}. For this purpose, we introduce for $q\in(1,2)$ and $s\in(1,\infty)$ the function space
\begin{align}\label{MR_DefOfxoseenq}
\begin{aligned}
&\xoseen{q,s}(\R^3):=\setcl{\vvel\in\LRloc{1}(\R^3)^3}{\oseennorm{\vvel}{q,s}<\infty},\\
&\oseennorm{\vvel}{q,s} := \norm{\vvel}_{\frac{2q}{2-q}} + \norm{\grad\vvel}_{\frac{4}{4-q}} + \norm{\partial_1\vvel}_q+\norm{\grad^2\vvel}_q+\norm{\grad^2\vvel}_s.  
\end{aligned}
\end{align}
The Banach space $\xoseen{q,s}(\R^3)$ is intrinsically linked to the classical three-dimensional Oseen operator. More specifically,
the norm $\oseennorm{\vvel}{q,s}$ captures precisely the generic integrability of the velocity field of a solution to the three-dimensional steady-state Oseen system.
To describe similar properties for the velocity field of a solution to the time-periodic Oseen system, we introduce for $q\in(1,\infty)$ the space
\begin{align}\label{DefOsWSRperspace}
\begin{aligned}
&\WSRper{1,2}{q}(\RthreeR) := \closure{\CRciper(\R\times\R^3)}{\norm{\cdot}_{1,2,q}},\\
&\norm{\uvel}_{1,2,q}:=\Bp{
\sum_{\snorm{\alpha}\leq 1} 
\iper\int_0^\per\int_{\R^3}\snorm{\partial_t^\alpha\uvel}^q\,\dx\dt 
+
\sum_{\snorm{\beta}\leq 2} 
\iper\int_0^\per\int_{\R^3}\snorm{\partial_x^\beta\uvel}^q\,\dx\dt}^{\frac{1}{q}}.
\end{aligned}
\end{align}
For $q\in(1,3)$ and $s\in(1,\infty)$ we further define
\begin{align}\label{MR_DefOfXpres}
\begin{aligned}
&\xpresper{q,s}\bp{\R\times\R^3}:=\closure{\CRciper(\R\times\R^3)}{\norm{\cdot}_{\xpresper{q,s}}},\\
&\norm{\upres}_{\xpresper{q,s}}:=\Bp{\int_0^\per \norm{\upres(t,\cdot)}_{\frac{3q}{3-q}}^q + \norm{\grad\upres(t,\cdot)}_q^q\,\dt }^{1/q}+\norm{\grad\upres}_{s},
\end{aligned}
\end{align}
which is the space we shall use to characterize the corresponding pressure term of the solution.

We first show that a physically reasonable weak time-periodic solution in the sense of Definition \ref{WeakSolClassDef} satisfying
the integrability condition \eqref{MainThm_condonsol} is in fact a strong solution. 
We then recall the regularity results for strong solutions from \cite{ertpns}, which yield that $\uvel$ is smooth when the data is smooth.

\begin{lem}\label{RegLem}
Let $\rey\neq 0$, $f\in\CRciper\bp{\R\times\R^3}^3$  and 
$\uvel$ be a physically reasonable weak $\per$-time-periodic solution to \eqref{intro_nspastbodywholespace} in the sense of 
Definition \ref{WeakSolClassDef}. If $\uvel$ satisfies \eqref{MainThm_condonsol}, then
\begin{align}\label{RegLemRegOfprojcomplu}
\forall q\in(1,\infty):\  \projcompl\uvel \in\WSRper{1,2}{q}\bp{\R\times\R^3}^3  
\end{align}
and
\begin{align}\label{RegLemRegOfproju}
\forall q\in(1,2),\, s\in(1,\infty):\ \proj\uvel\in \xoseen{q,s}(\R^3).
\end{align}
Moreover, there is a pressure term 
\begin{align}\label{RegLemRegOfPres}
\forall q\in(1,3),\, s\in(1,\infty):\ \upres\in \xpresper{q,s}\bp{\R\times\R^3} 
\end{align}
such that $(\uvel,\upres)$ is a solution to \eqref{intro_nspastbodywholespace}. Additionally, 
\begin{align}\label{RegLemSmoothnes}
(\uvel,\upres)\in\CRi(\R\times\R^3)^3\times \CRi(\R\times\R^3).
\end{align}
\end{lem}

\begin{proof}
Put $\vvel:=\proj\uvel$ and $\wvel:=\projcompl\uvel$. 
From Definition \ref{WeakSolClassDef} and standard Sobolev embedding, we infer 
\begin{align}
&\vvel\in\LR{6}(\R^3)^3,\quad
\grad\vvel\in\LR{2}(\R^3)^{3\times 3}, \label{RegLem_IniRegvvel}\\
&\wvel\in\LR{\infty}\bp{(0,\per);\LR{2}(\R^3)^3}\cap\LR{2}\bp{(0,\per);\LR{6}(\R^3)^3},\quad
\grad\wvel\in\LR{2}\bp{(0,\per);\LR{2}(\R^3)^{3\times 3}}.\label{RegLem_IniRegwvel}
\end{align}
It is easy to verify from \eqref{WeakSolClassDefDefofweaksol} that $\vvel$ is a generalized solution 
(sometimes  referred to as a \emph{D-solution}) to the steady-state Navier-Stokes problem
\begin{align}\label{RegLem_steadystateeq}
\begin{pdeq}
&-\Delta\vvel-\rey \partial_1\vvel+\grad\vpres + \nsnonlin{\vvel} = \proj f - \proj\bb{\nsnonlin{\wvel}} && \tin\ \R^3,\\
&\Div\vvel =0 && \tin\ \R^3,\\
&\lim_{\snorm{x}\ra\infty} \vvel= 0.
\end{pdeq}
\end{align}
Problem \eqref{RegLem_steadystateeq} has been studied extensively over the years. A comprehensive exposition can be found in \cite[Chapter IX]{galdi:book2}.
Although \cite[Chapter IX]{galdi:book2} deals with the exterior domain case, virtually all techniques found there can be used in the whole-space case as well.
A straightforward utilization of the integral Minkowski inequality in combination with \eqref{RegLem_IniRegwvel} yields 
$\proj\bb{\nsnonlin{\wvel}}\in\LR{1}(\R^3)^{3}\cap\LR{\frac{3}{2}}(\R^3)^{3}$. 
From \cite[Chapter IX]{galdi:book2} it thus follows that
\begin{align}\label{RegLem_vvelL3}
\vvel\in\LR{3}(\R^3)^3.
\end{align}
Now consider the linear operator
\begin{align*}
\Oseensolopr : \SR(\R^3)^3 \ra \TDR(\R^3)^3,\quad \Oseensolopr\psi := \iFT_{\R^3}\Bb{\frac{1}{\snorm{\xi}^2-i\rey\xi_1}\cdot\,\FT_{\R^3}\bb{\hproj\psi}}
\end{align*}
One may recognize $\Oseensolopr$ as the inverse to the Oseen operator. By a well-known application of Lizorkin's multiplier theorem, see 
\cite[Theorem 4.1]{galdi:book1}, one can show that $\Oseensolopr$ extends uniquely to a homeomorphism
\begin{align}\label{RegLem_Oseensolmapprop1}
\forall q\in(1,2),s\in(1,\infty):\quad \Oseensolopr: \LR{q}(\R^3)^3\cap\LR{s}(\R^3)^3\ra\xoseen{q,s}(\R^3).
\end{align}
Additionally, see \cite[Theorem 4.2]{galdi:book1}, the composition $\Oseensolopr\circ\Div$ extends to a continuous linear operator
\begin{align}\label{RegLem_Oseensolmapprop2}
\begin{aligned}
&\forall q\in(1,4):\quad \Oseensolopr\circ\Div: \LR{q}(\R^3)^{3\times 3}\ra\LR{\frac{4q}{4-q}}(\R^3)^3,\\
&\forall q\in(1,\infty):\quad \Oseensolopr\circ\Div: \LR{q}(\R^3)^{3\times 3}\ra\DSR{1}{q}(\R^3)^3.
\end{aligned}
\end{align}
We shall employ \eqref{RegLem_Oseensolmapprop1} and \eqref{RegLem_Oseensolmapprop2} iteratively to establish increased integrability of $\vvel$.
Observe that
\begin{align}\label{RegLem_RepOfvvel}
\begin{aligned}
\vvel &= \Oseensolopr\bb{\proj f - \proj\bb{\nsnonlin{\wvel}} - \nsnonlin{\vvel}}\\
&= \Oseensolopr\proj f - \Oseensolopr\circ\Div\proj\bb{\wvel\otimes\wvel} - \Oseensolopr\circ\Div\bb{\vvel\otimes\vvel}.
\end{aligned}
\end{align}
The first equality above follows by a standard uniqueness argument since both sides of the equation is a solution in $\LR{6}(\R^3)^3$ to the same Oseen system.
One may verify the second equality by a simple computation. It follows from \eqref{RegLem_IniRegwvel} that
$\proj\bb{\wvel\otimes\wvel}\in\LR{1}(\R^3)^{3\times 3}\cap\LR{3}(\R^3)^{3\times 3}$. By \eqref{RegLem_Oseensolmapprop1} and \eqref{RegLem_Oseensolmapprop2}
we thus have
\begin{align}\label{RegLem_rhsreg1}
\forall q\in(2,12]:\quad \Oseensolopr\proj f - \Oseensolopr\circ\Div\proj\bb{\wvel\otimes\wvel}\in\LR{q}(\R^3).
\end{align}
By \eqref{RegLem_IniRegvvel} and \eqref{RegLem_vvelL3} we have $\vvel\otimes\vvel\in\LR{\frac{3}{2}}(\R^3)^{3\times 3}\cap\LR{3}(\R^3)^{3\times 3}$, whence 
\eqref{RegLem_Oseensolmapprop2}, \eqref{RegLem_RepOfvvel} and \eqref{RegLem_rhsreg1} imply 
$\vvel\in\LR{\frac{12}{5}}(\R^3)^3\cap\LR{12}(\R^3)^3$. From this, $\vvel\otimes\vvel\in\LR{\frac{6}{5}}(\R^3)^{3\times 3}\cap\LR{6}(\R^3)^{3\times 3}$ follows.
Consequently, \eqref{RegLem_Oseensolmapprop2}, \eqref{RegLem_RepOfvvel} and \eqref{RegLem_rhsreg1} then yield
\begin{align}\label{RegLem_FinalRegvvel}
\forall q\in (2,12]:\quad \vvel\in\LR{q}(\R^3)^3.
\end{align}
We now turn our attention to the integrability of $\wvel$. 
It follows directly from the definitions \eqref{DefOfAnisotropicSobSpaceOnGrp} and \eqref{DefOsWSRperspace} that lifting by $\quotientmap$ is a homeomorphism between $\WSR{1,2}{q}(\grp)$ and $\WSRper{1,2}{q}\np{\RthreeR}$.
We may therefore consider $\wvel$ as a function defined on $\grp$. It is easy to verify from \eqref{WeakSolClassDefDefofweaksol} that $\wvel$ is
a solution in the sense of $\TDR(\grp)$-distributions to 
\begin{align}\label{RegLem_eqforwvel}
\begin{pdeq}
&\partial_t\wvel -\Delta\wvel -\rey\partial_1\wvel  = 
\projcompl\hproj f - \projcompl\hproj\bb{\nsnonlin{\wvel}+\nsnonlinb{\wvel}{\vvel}+\nsnonlinb{\vvel}{\wvel}} && \tin\grp,\\
&\Div\wvel =0 && \tin\grp.
\end{pdeq}
\end{align} 
Observe that the integrability of the right-hand side above is sufficient for the Helmholtz projection $\hproj$ hereof to be well-defined. In fact, due to 
the assumption that $\wvel$ satisfies \eqref{MainThm_condonsol} it is easy to verify that all terms on the right-hand side belong to some $\LR{q}(\grp)$ space.
Consequently, we can further employ Lemma \ref{Reformulation_OseensoloprcomplMappingProps} and Lemma \ref{Reformulation_UniquenessComplEq} to deduce
\begin{align}
\wvel 
&= \Oseensoloprcompl \hproj f - {\Oseensoloprcompl\hproj\nb{\nsnonlin{\wvel}}-
\Oseensoloprcompl\hproj\nb{\nsnonlinb{\wvel}{\vvel}}-\Oseensoloprcompl\hproj\nb{\nsnonlinb{\vvel}{\wvel}}}\label{RegLem_RepOfwvelPre}\\
&= \Oseensoloprcompl \hproj f - {\Oseensoloprcompl\circ\Div\hproj\nb{\wvel\otimes\wvel-\vvel\otimes\wvel-\wvel\otimes\vvel}}\label{RegLem_RepOfwvelPreDivForm}.
\end{align}
More precisely, due to \eqref{Reformulation_OseensoloprcomplMappingPropsProp1} the right-hand side in \eqref{RegLem_RepOfwvelPre} is well-defined as a solution
to \eqref{RegLem_eqforwvel} in $\projcompl\TDR(\grp)$. From Lemma \ref{Reformulation_UniquenessComplEq} we thus obtain the identity in \eqref{RegLem_RepOfwvelPre},
and in turn, by an easy calculation, \eqref{RegLem_RepOfwvelPreDivForm}.
For $q\in(1,10)$, the following implication follows from Lemma \ref{Reformulation_OseensoloprcomplMappingProps}:
\begin{align*}
w\in\LR{q}(\grp)^3 \Ra \hproj\nb{\wvel\otimes\wvel}\in\LR{\frac{q}{2}}(\grp)^{3\times 3} \Ra \Oseensoloprcompl\circ\Div\hproj\nb{\wvel\otimes\wvel}\in\LR{\frac{5q}{10-q}}(\grp).
\end{align*}
Recalling \eqref{RegLem_FinalRegvvel}, we also deduce for $q\in(1,60/7)$ validity of the implication
\begin{align*}
w\in\LR{q}(\grp)^3 \Ra \hproj\nb{\vvel\otimes\wvel}\in\LR{\frac{12q}{12+q}}(\grp)^{3\times 3} \Ra \Oseensoloprcompl\circ\Div\hproj\nb{\vvel\otimes\wvel}
\in\LR{\frac{60q}{60-7q}}(\grp).
\end{align*}
Employing Lemma \ref{Reformulation_OseensoloprcomplMappingProps} in a similar manner also for $q\in[60/7,\infty)$, we thus deduce from \eqref{RegLem_RepOfwvelPreDivForm} the implication
\begin{align}\label{RegLem_FundImplicationNew}
\forall q\in(1,\infty):\ \wvel\in\LR{q}(\grp)^3\ \Ra\ 
\begin{pdeq}
&\wvel\in\LR{\frac{5q}{10-q}}(\grp) && \tif q\in[5,6],\\
&\wvel\in\LR{\frac{60q}{60-7q}}(\grp) && \tif q\in[6,60/7),\\
&\forall s\in[q,\infty):\ \wvel \in \LR{s}(\grp) && \tif q=60/7,\\
&\wvel\in\LR{\infty}(\grp) && \tif q\in(60/7,\infty).
\end{pdeq}
\end{align}
By assumption \eqref{MainThm_condonsol}, $\wvel\in\LR{r}(\grp)^3$. Since $r>5$, iterating \eqref{RegLem_FundImplicationNew} with $q=r$ as starting point 
yields 
\begin{align}\label{RegLem_FinalLinftyRegwvel}
\wvel\in\LR{\infty}(\grp)^3.
\end{align}
With this improved integrability of $\wvel$, we return to \eqref{RegLem_RepOfvvel} and improve \eqref{RegLem_FinalRegvvel} 
to 
\begin{align}\label{RegLem_FinalFinalRegvvel}
\forall q\in (2,\infty):\quad \vvel\in\LR{q}(\R^3)^3.
\end{align}
From \eqref{RegLem_Oseensolmapprop2}, \eqref{RegLem_RepOfvvel}, \eqref{Reformulation_OseensoloprcomplMappingPropsProp2} and \eqref{RegLem_RepOfwvelPreDivForm}
it now follows that
\begin{align}\label{RegLem_FinalRegGradvvelGradwvel}
\forall q\in(1,\infty):\quad \grad\vvel\in\LR{q}(\R^3)^{3\times 3},\ \grad\wvel\in\LR{q}(\grp)^{3\times 3}.
\end{align}
Combining \eqref{RegLem_FinalLinftyRegwvel}, \eqref{RegLem_FinalFinalRegvvel}, \eqref{RegLem_FinalRegGradvvelGradwvel} and \eqref{RegLem_RepOfwvelPre}, we conclude
\eqref{RegLemRegOfprojcomplu} by Lemma \ref{Reformulation_OseensoloprcomplMappingProps}.
Similarly, recalling \eqref{RegLem_Oseensolmapprop2} and \eqref{RegLem_RepOfvvel}, we deduce \eqref{RegLemRegOfproju}.

The integrability we have established for $\vvel$ and $\wvel$ enables us to employ the Helmholtz-Weyl projection $\hproj$ in \eqref{intro_nspastbodywholespace} and
obtain the equation
\begin{align*}
\grad\upres = \bp{\id-\hproj}\bb{f-\nsnonlin{\wvel}-\nsnonlinb{\wvel}{\vvel}-\nsnonlinb{\vvel}{\wvel}-\nsnonlin{\vvel}}\quad\tin\grp
\end{align*}
for the pressure term $\upres$. Applying the Fourier transform $\FT_{\R^3}$, we thus formally arrive at the following expression for $\upres$:
\begin{align*}
\upres=\iFT_{\R^3}\Bb{\frac{\xi_j}{\snorm{\xi}^2}\cdot \FT_{\R^3}\bb{f-\nsnonlin{\wvel}-\nsnonlinb{\wvel}{\vvel}-\nsnonlinb{\vvel}{\wvel}-\nsnonlin{\vvel}}}.
\end{align*}
By well-known embedding properties of the Riesz potential (see for example \cite[Theorem 6.1.3]{Grafakos2}) and $\LR{p}$-boundedness of the Riesz operators,
we deduce from the integrability properties of $\vvel$ and $\wvel$ that the expression for $\upres$ is valid as an element of $\xpresper{q,s}\bp{\R\times\R^3}$
for any $q\in(1,3)$ and $s\in(1,\infty)$. This concludes \eqref{RegLemRegOfPres}. Finally, we obtain from \cite[Corollary 2.5]{ertpns} that both $\uvel$ and $\upres$
are smooth.
\end{proof} 

The smoothness of the solution $\uvel$ in Lemma \ref{RegLem} enables us to analyze it's pointwise behavior. We start by showing 
a decay estimate of $\grad\uvel$. The lemma is a minor modification of \cite[Lemma IX.8.2]{galdi:book2}.

\begin{lem}\label{ae_firstlemma}
Let $f\in\CRciper\bp{\R\times\R^3}^3$ and $(\uvel,\upres)$ be a $\per$-time-periodic solution to \eqref{intro_nspastbodywholespace} with 
\begin{align}\label{ae_firstlemmaCondOnSol}
\begin{aligned}
&\uvel\in\CR{2}\np{\R\times\R^3}^3,\quad\upres\in\CR{1}(\R\times\R^3),\\
&\grad\uvel\in\LR{2}\bp{\R\times\R^3}^{3\times 3}\cap\LR{\frac{3}{2}}\bp{\R\times\R^3}^{3\times 3},\\
&\forall s\in(2,3]:\ \uvel\in\LR{s}(\grp)^3,\quad\exists s'\in[3/2,2):\ \upres\in\LR{{s'}}(\grp).
\end{aligned}
\end{align}
Then
\begin{align}\label{ae_firstlemmaDecay}
\forall\epsilon>0:\quad \iper\int_0^\per\int_{\B^R} \snorm{\grad\uvel}^2\,\dx\dt \leq \Cc[ae_firstlemmaConst]{C}\,R^{-1+\epsilon}.
\end{align}
\end{lem}
\begin{proof}\newCCtr[c]{ae_firstlemma}
The proof follows that of \cite[Lemma IX.8.2]{galdi:book2}. We put 
\begin{align*}
\calg(R):=\iper\int_0^\per\int_{\B^R} \snorm{\grad\uvel}^2\,\dx\dt = \iper\int_0^\per\int_R^\infty\int_{\partial\B_r} \snorm{\grad\uvel}^2\,\dS\dr\dt.
\end{align*}
By assumption $\grad\uvel\in\CR{}(\R\times\R^3)$, whence $\calg\in\CR{1}(0,\infty)$.
Multiplying both sides in \eqref{intro_nspastbodywholespace} by $\uvel$ and subsequently integrating over $\B_{R^*,R}\times[0,\per]$, we obtain
\begin{align}\label{ae_firstlemmaAfterTest}
\iper\int_0^\per\int_{\B_{R^*,R}} \snorm{\grad\uvel}^2\,\dx\dt = 
\iper\int_0^\per\int_{\partial\B_{R^*,R}} \frac{\rey}{2}\snorm{\uvel}^2 n_1 - \half\snorm{\uvel}^2\uvel\cdot n + \uvel\cdot\bp{\grad\uvel\cdot n}- \upres\,\uvel\cdot n\,\dS\dt. 
\end{align}
From the assumptions in \eqref{ae_firstlemmaCondOnSol}, it follows that the function
\begin{align}\label{ae_firstlemmainitialL1func}
r\ra\iper\int_0^\per\int_{\partial\B_{r}} \snorm{\uvel}^3 + \snormL{-\half\snorm{\uvel}^2\uvel\cdot n + \uvel\cdot\bp{\grad\uvel\cdot n} - \upres\,\uvel\cdot n}\,\dS\dt 
\end{align}
belongs to $\LR{1}\bp{(1,\infty)}$. Consequently, there is a sequence $\seqK{R_k}\subset(0,\infty)$ with $\lim_{k\ra\infty}R_k=\infty$
and 
\begin{align*}
\lim_{k\ra\infty} R_k\,\iper\int_0^\per\int_{\partial\B_{R_k}} \snorm{\uvel}^3 + \snormL{-\half\snorm{\uvel}^2\uvel\cdot n + \uvel\cdot\bp{\grad\uvel\cdot n}- \upres\,\uvel\cdot n}\,\dS\dt =0.
\end{align*}
Employing H\"older's inequality, we observe that 
\begin{align}\label{ae_firstlemmausquareest}
\iper\int_0^\per\int_{\partial\B_R} \snorm{\uvel}^2\,\dS\dt 
\leq \Cc{ae_firstlemma} \iper\int_0^\per \Bp{\int_{\partial\B_R} \snorm{\uvel}^3\,\dS}^{\frac{2}{3}} R^{\frac{2}{3}}\,\dt 
\leq \Cclast{ae_firstlemma}\, \Bp{ R \iper\int_0^\per\int_{\partial\B_R}\snorm{\uvel}^3\,\dS\dt}^{\frac{2}{3}}.
\end{align}
It follows that 
\begin{align*}
\lim_{k\ra\infty} \iper\int_0^\per\int_{\partial\B_{R_k}} \frac{\rey}{2}\snorm{\uvel}^2 n_1 - \half\snorm{\uvel}^2\uvel\cdot n + \uvel\cdot\bp{\grad\uvel\cdot n} - \upres\,\uvel\cdot n\,\dS\dt = 0. 
\end{align*}
Put
\begin{align*}
F(R):=\iper\int_0^\per\int_{\partial\B_{R}} \frac{\rey}{2}\snorm{\uvel}^2 n_1 - \half\snorm{\uvel}^2\uvel\cdot n + \uvel\cdot\bp{\grad\uvel\cdot n} - \upres\,\uvel\cdot n\,\dS\dt. 
\end{align*}
Choosing $R^*=R_k$ in \eqref{ae_firstlemmaAfterTest} and letting $k\ra\infty$, we see that
\begin{align}\label{ae_firstlemmacalgeqF}
\calg(R)=F(R).
\end{align}
By a similar estimate as in \eqref{ae_firstlemmausquareest}, we find for $q\in(1,\infty)$
\begin{align*}
\iper\int_0^\per\int_{\partial\B_R} \snorm{\uvel}^2\,\dS\dt \leq \Cc{ae_firstlemma}\,R^{\frac{2(q-1)}{q}}  
\Bp{\iper\int_0^\per\int_{\partial\B_R}\snorm{\uvel}^{2q}\,\dS\dt}^\frac{1}{q}.
\end{align*}
Employing Young's inequality, we then deduce
\begin{align}\label{ae_firstlemmausquareRepsest}
R^{-\epsilon} \iper\int_0^\per\int_{\partial\B_R} \snorm{\uvel}^2\,\dS\dt 
\leq \Cc{ae_firstlemma}\Bp{ R^{2-\epsilon\frac{q}{q-1}} +  
\iper\int_0^\per\int_{\partial\B_R}\snorm{\uvel}^{2q}\,\dS\dt}.
\end{align}
Now fix $q\in (1,\frac{3}{3-\epsilon})$. Then $2-\epsilon\frac{q}{q-1}<-1$.
Moreover, by assumption \eqref{ae_firstlemmaCondOnSol} we have $\uvel\in\LR{2q}\bp{\R\times\R^3}^3$.
Thus
\begin{align*}
R\ra R^{-\epsilon} \iper\int_0^\per\int_{\partial\B_R} \snorm{\uvel}^2\,\dS\dt \in\LR{1}\bp{(1,\infty)}.
\end{align*}
This information together with the summability of the function in \eqref{ae_firstlemmainitialL1func} implies
that $R\ra R^{-\epsilon}F(R)\in\LR{1}\bp{(1,\infty)}$.
Recalling \eqref{ae_firstlemmacalgeqF}, we deduce
\begin{align}\label{ae_firstlemmaFinalSummability}
R\ra R^{-\epsilon}\calg(R)\in\LR{1}\bp{(1,\infty)}.
\end{align} 
Combining \eqref{ae_firstlemmaFinalSummability} with the fact that
\begin{align*}
\calg'(R) = - \iper\int_0^\per\int_{\partial\B_R} \snorm{\grad\uvel}^2\,\dS\dt \leq 0,
\end{align*}
we finally conclude, by \cite[Lemma IX.8.1]{galdi:book2}, that
\begin{align*}
\calg(R)R^{1-\epsilon}\leq \Cc{ae_firstlemma},
\end{align*}
which yields \eqref{ae_firstlemmaDecay}
\end{proof}

We observe in the following corollary that the decay property established for $\uvel$ in Lemma \ref{ae_firstlemma} also holds
for $\proj\uvel$ and $\projcompl\uvel$ separately. 
\begin{cor}\label{ae_firstcor}
Under the same assumptions as in Lemma \ref{ae_firstlemma}, it holds that 
\begin{align}\label{ae_firstcorDecay}
\forall\epsilon>0:\quad \iper\int_0^\per\int_{\B^R} \snorm{\grad\projcompl\uvel}^2\,\dx\dt + \int_{\B^R} \snorm{\grad\proj\uvel}^2\,\dx \leq \Cc[ae_firstcorConst]{C}\,R^{-1+\epsilon}.
\end{align}
\end{cor}
\begin{proof}
We simply observe that
\begin{align*}
\begin{aligned}
&\iper\int_0^\per\int_{\B^R} \snorm{\grad\uvel}^2\,\dx\dt \\
&\qquad = 
\iper\int_0^\per\int_{\B^R} \snorm{\grad\projcompl\uvel}^2\,\dx\dt +
\int_{\B^R} \snorm{\grad\proj\uvel}^2\,\dx + 
2\iper\int_0^\per\int_{\B^R} \grad\proj\uvel:\grad\projcompl\uvel\,\dx\dt
\end{aligned}
\end{align*}
and 
\begin{align*}
\iper\int_0^\per\int_{\B^R} \grad\proj\uvel:\grad\projcompl\uvel\,\dx\dt=
\int_{\B^R} \grad\proj\uvel:\Bp{\iper\int_0^\per\grad\projcompl\uvel\,\dt}\,\dx = 0.
\end{align*}
\end{proof}

Next, we recall some well-known properties of the Oseen fundamental solution $\fundsoloseen$.
\begin{lem}\label{ae_oseenfundsolprops}
Let $\rey\neq 0$. The Oseen fundamental solution $\fundsoloseen$ satisfies
\begin{align}
&\forall \snorm{x}>0:\quad \snorml{\fundsoloseen(x)}\leq\Cc{C}\,{\snorm{x}^{-1}}, \label{ae_fundsoloseenpdecay}\\
&\forall \snorm{x}>1:\quad\snorml{\grad\fundsoloseen(x)}\leq \Cc{C}\,{\snorm{x}^{-\frac{3}{2}}} \label{ae_fundsoloseengradpdecay},\\
&\forall r>0:\quad\int_{\partial\B_r}\snorml{\grad\fundsoloseen(x)}\,\dS\leq \Cc{C}\,r^{-\half}\label{ae_fundsoloseenintgradpdecay}.
\end{align}
Moreover, $\fundsoloseen$ enjoys for any $r>0$ the summability properties
\begin{align}
&\forall q\in(2,\infty): &&\fundsoloseen\in\LR{q}(\R^3\setminus\B_r)^{3\times 3},\label{ae_oseenfundsolprops_SummabilityFundSol}\\
&\forall q\in [1,3):&& \fundsoloseen\in\LRloc{q}(\R^3)^{3\times 3},\label{ae_oseenfundsolprops_LocSummabilityFundSol}\\
&\forall q\in(4/3,\infty):&& \partial_j\fundsoloseen\in\LR{q}(\R^3\setminus\B_r)^{3\times 3}\quad(j=1,2,3),\label{ae_oseenfundsolprops_SummabilityFundSolGrad}\\
&\forall q\in [1,3/2):&& \partial_j\fundsoloseen\in\LRloc{q}(\R^3)^{3\times 3}\quad(j=1,2,3).\label{ae_oseenfundsolprops_LocSummabilityFundSolGrad}
\end{align}   
\end{lem}
\begin{proof}
We refer to \cite[Chapter VII: (3.24), (3.28), (3.32), (3.33)]{galdi:book1} for 
\eqref{ae_fundsoloseenpdecay}, 
\eqref{ae_oseenfundsolprops_SummabilityFundSol},
\eqref{ae_fundsoloseengradpdecay}, and
\eqref{ae_oseenfundsolprops_SummabilityFundSolGrad}.
Estimate \eqref{ae_fundsoloseenintgradpdecay} follows from \cite[Exercise VII.3.1]{galdi:book1}. Finally, 
\eqref{ae_oseenfundsolprops_LocSummabilityFundSol} and \eqref{ae_oseenfundsolprops_LocSummabilityFundSolGrad} are direct consequences of 
\eqref{ae_fundsoloseenpdecay} and \eqref{ae_fundsoloseenintgradpdecay}, respectively. 
\end{proof}

We are now in a position to prove the main theorem.

\begin{proof}[Proof of Theorem \ref{MainThm}]\newCCtr[c]{ae_asympexpansionthm}

We start by associating to the weak solution $\uvel$ the pressure term $\upres$ from Lemma \ref{RegLem}. By Lemma \ref{RegLem},
$(\uvel,\upres)$ is then a smooth strong solution. Additionally, the 
integrability properties \eqref{RegLemRegOfprojcomplu}--\eqref{RegLemRegOfPres} obtained in Lemma \ref{RegLem} ensure that 
$(\uvel,\upres)$ satisfies \eqref{ae_firstlemmaCondOnSol}. We can therefore utilize Lemma \ref{ae_firstlemma} and Corollary \ref{ae_firstcor}. 
Further observe that since $\fundsoloseen$ is smooth away from the origin,
it is enough to show the estimate in \eqref{MainThm_estofrestterm} for large $\snorm{x}$.  
Consider therefore an arbitrary $x\in\R^3$ with $\snorm{x}>2$. Let $R:=\frac{\snorm{x}}{2}$.
Put $\vvel:=\proj\uvel$ and $\wvel:=\projcompl\uvel$.
We will again take advantage of the fact that $\wvel$ satisfies \eqref{RegLem_eqforwvel} and, as in the proof of Lemma \ref{RegLem},
use Lemma \ref{Reformulation_OseensoloprcomplMappingProps} and Lemma \ref{Reformulation_UniquenessComplEq} to deduce \eqref{RegLem_RepOfwvelPre}.
By Lemma \ref{ae_integrabilityoffundsolcompl}, the componentwise convolution $\fundsolvelcompl*\psi$
is well-defined in a classical sense for any $\psi\in\LR{q}(\grp)^3$, $q\in[1,\infty)$. Since clearly $\Oseensoloprcompl\hproj\psi=\fundsolvelcompl*\psi$, we thus 
have
\begin{align}\label{ae_asympexpansionthmrepofwvel}
\wvel = \fundsolvelcompl*\Bp{f-\nsnonlin{\wvel}-\nsnonlinb{\wvel}{\vvel}-\nsnonlinb{\vvel}{\wvel}}.
\end{align}
Put
\begin{align}\label{ae_asympexpansionthmSplitofwvel}
\begin{aligned}
&I_1(t,x):=\fundsolvelcompl*\bb{\nsnonlin{\wvel}}(t,x),\\
&I_2(t,x):=\fundsolvelcompl*\bb{\nsnonlinb{\wvel}{\vvel}}(t,x),\\
&I_3(t,x):=\fundsolvelcompl*\bb{\nsnonlinb{\vvel}{\wvel}}(t,x),\\
&I_4(t,x):=\fundsolvelcompl* f (t,x).
\end{aligned}
\end{align} 
We shall give a pointwise estimate of $I_1,I_2,I_3$, and $I_4$. We first split 
\begin{align*}
\begin{aligned}
I_1(t,x)&=\iper\int_0^\per\int_{\B_R}\fundsolvelcompl(x-y,t-s)\,\bb{\nsnonlin{\wvel}}(y,s)\,\dy\ds\\
&\quad +\iper\int_0^\per\int_{\B^R}\fundsolvelcompl(x-y,t-s)\,\bb{\nsnonlin{\wvel}}(y,s)\,\dy\ds\\
&=: I_{11} + I_{12}.
\end{aligned}
\end{align*}
Employing H\"older's inequality and recalling \eqref{ae_estoffundsolcomplest}, we deduce 
\begin{align*}
\begin{aligned}
\snorm{I_{11}}&\leq  \int_{\B_R} \Bp{\iper\int_0^\per\snorml{\fundsolvelcompl(x-y,s)}^2\,\ds}^\half \Bp{\iper\int_0^\per\snorml{\nsnonlin{\wvel}}^2\,\ds}^\half \,\dy\\
&\leq \Cc{ae_asympexpansionthm} \frac{1}{R^3}\, \Bp{\int_{\B_R} 1\,\dy}^\half \Bp{\iper\int_0^\per\int_{\B_R}\snorml{\nsnonlin{\wvel}}^2\,\dy\ds}^\half\\
&\leq \Cc{ae_asympexpansionthm} R^{-\frac{3}{2}},
\end{aligned}
\end{align*}
where we in the last inequality use that $\nsnonlin{\wvel}\in\LR{2}(\grp)^3$, which is a direct consequence of Lemma \ref{RegLem}. 
Consider $\alpha\in(1,\frac{5}{3})$ and let $\alpha'$ denote the corresponding H\"older conjugate.
Recalling Corollary \ref{ae_firstcor}, we find that  
\begin{align*}
\begin{aligned}
\snorm{I_{12}}&\leq \Bp{\iper\int_0^\per\int_{\B^R}\snorml{\fundsolvelcompl(x-y,s)}^\alpha\,\dy\ds}^{\frac{1}{\alpha}}
\Bp{\iper\int_0^\per\int_{\B^R}\snorm{\wvel}^{\alpha'}\snorm{\grad\wvel}^{\alpha'}\,\dy\ds}^{\frac{1}{\alpha'}}\\
&\leq \Cc{ae_asympexpansionthm}\, \norm{\fundsolvelcompl}_\alpha\, \norm{\wvel}_{\LR{\infty}\np{\Torus\times\B^R}}\, \norm{\grad\wvel}_{\infty}^{\frac{\alpha'-2}{\alpha'}}
\,\Bp{\iper\int_0^\per\int_{\B^R}\snorm{\grad\wvel}^{2}\,\dy\ds}^{\frac{1}{\alpha'}}\\
&\leq \Cc{ae_asympexpansionthm}\, \norm{\fundsolvelcompl}_\alpha\, \norm{\wvel}_{\LR{\infty}\np{\Torus\times\B^R}}\, \norm{\grad\wvel}_{\infty}^{\frac{2}{\alpha}-1}
\,R^{\np{-1+\epsilon}\frac{\alpha-1}{\alpha}}
\end{aligned}
\end{align*}
for all $\epsilon>0$.
By Sobolev embedding, see \cite[Theorem 4.1]{tpfpb}, $\grad\wvel\in\LR{\infty}(\grp)$. Recalling \eqref{ae_integrabilityoffundsolcomplsummability} and letting $\alpha\ra\frac{5}{3}$, we thus obtain 
\begin{align*}
\begin{aligned}
\forall\epsilon>0:\quad \snorm{I_{12}}&\leq \Cc{ae_asympexpansionthm}\, R^{-\frac{2}{5}+\epsilon}\, \norm{\wvel}_{\LR{\infty}\np{\Torus\times\B^R}}.
\end{aligned}
\end{align*}
We thus conclude
\begin{align}\label{ae_asympexpansionthmestofI1final}
\begin{aligned}
\forall\epsilon>0:\quad \snorm{I_{1}}&\leq \Cc{ae_asympexpansionthm}\, \bp{R^{-\frac{3}{2}}+R^{-\frac{2}{5}+\epsilon}\, \norm{\wvel}_{\LR{\infty}\np{\Torus\times\B^R}}}.
\end{aligned}
\end{align}
Similarly, we estimate
\begin{align}
&\forall\epsilon>0:\quad \snorm{I_{2}}\leq \Cc{ae_asympexpansionthm}\, \bp{R^{-\frac{3}{2}}+R^{-\frac{2}{5}+\epsilon}\, \norm{\wvel}_{\LR{\infty}\np{\Torus\times\B^R}}}\label{ae_asympexpansionthmestofI2final},\\
&\forall\epsilon>0:\quad \snorm{I_{3}}\leq \Cc{ae_asympexpansionthm}\, \bp{R^{-\frac{3}{2}}+R^{-\frac{2}{5}+\epsilon}\, \norm{\vvel}_{\LR{\infty}\np{\B^R}}}\label{ae_asympexpansionthmestofI3final}.
\end{align}
Due to \eqref{ae_estoffundsolcomplest} and the fact that $\supp f$ is compact, we deduce
\begin{align}\label{ae_asympexpansionthmestofI4final}
\begin{aligned}
\snorm{I_4}\leq \Cc{ae_asympexpansionthm} \int_{\supp f} \Bp{\iper\int_0^\per\snorml{\fundsolvelcompl(x-y,s)}^2\,\ds}^\half \,\dy
\leq \Cc{ae_asympexpansionthm} R^{-3}
\end{aligned}
\end{align}
for $R$ sufficiently large. 
It follows from \eqref{ae_asympexpansionthmrepofwvel} and \eqref{ae_asympexpansionthmestofI1final}-\eqref{ae_asympexpansionthmestofI4final} that 
\begin{align}\label{ae_asympexpansionthmestofwvela}
\forall\epsilon>0:\quad \snorm{\wvel(t,x)}\leq \Cc{ae_asympexpansionthm}\, 
\bp{R^{-\frac{3}{2}}+R^{-\frac{2}{5}+\epsilon}
(\norm{\vvel}_{\LR{\infty}\np{\B^R}} +\norm{\wvel}_{\LR{\infty}\np{\Torus\times\B^R}}) }.
\end{align}
We now turn our attention to $\vvel$. 
To establish a pointwise estimate of $\vvel$, we follow essentially the proof of \cite[Theorem IX.8.1]{galdi:book2}.
Recall that $\vvel$ satisfies \eqref{RegLem_steadystateeq}.
We claim that 
\begin{align}\label{ae_asympexpansionthmRepofvvel}
\vvel=\fundsoloseen*\bp{\proj f - \proj\bb{\nsnonlin{\wvel}} -\nsnonlin{\vvel} }.
\end{align}
The summability properties \eqref{ae_oseenfundsolprops_SummabilityFundSol} and \eqref{ae_oseenfundsolprops_LocSummabilityFundSol} of the 
Oseen fundamental solution $\fundsoloseen$ combined with the summability properties obtained for $\vvel$ and $\wvel$ in Lemma \ref{RegLem} implies that 
the convolution on the right-hand side above is well-defined in a classical sense.
Since both sides in \eqref{ae_asympexpansionthmRepofvvel} satisfy \eqref{RegLem_steadystateeq}, a standard uniqueness argument yields the identity.
Recalling \eqref{ae_fundsoloseenpdecay}
and that $f$ has compact support,
we see for $R$ sufficiently large that
\begin{align*}
\snorml{\fundsoloseen*\proj f(x)}  \leq \Cc{ae_asympexpansionthm} \int_{\supp \proj f} \frac{1}{\snorm{x-y}} \,\dy 
\leq \Cc{ae_asympexpansionthm} R^{-1}. 
\end{align*}
Moreover, we can estimate
\begin{align*}
\begin{aligned}
\snorml{\fundsoloseen*\proj\bb{\nsnonlin{\wvel}}(x)}
&=\snormL{\int_{\R^3} \fundsoloseen(x-y)\, \iper\int_0^\per \nsnonlin{\wvel(y,t)}\,\dt\,\dy}\\
&\leq \iper\int_0^\per\int_{\B_R} \frac{\Cc{ae_asympexpansionthm}}{R} \snorm{\nsnonlin{\wvel}}\,\dy\dt +
\iper\int_0^\per\int_{\B^R} \frac{\snorm{\wvel(y,t)}}{\snorm{x-y}} \snorm{\grad\wvel(y,t)}\,\dy\dt\\
&\leq \Cc{ae_asympexpansionthm} R^{-1} +
\iper\int_0^\per \Bp{\int_{\B^R} \frac{\snorm{\wvel(y,t)}^2}{\snorm{x-y}^2}\,\dy}^\half \Bp{\int_{\B^R}\snorm{\grad\wvel(y,t)}^2\,\dy}^\half\dt,
\end{aligned}
\end{align*}
where we use in the last estimate that $\nsnonlin{\wvel}\in\LR{1}(\grp)$.
We can use the Hardy-type inequality \cite[Theorem II.5.1]{galdi:book2} to increase the remaining integral on the right-hand side above by
\begin{align*}
\begin{aligned}
\iper\int_0^\per \Bp{\int_{\B^R} \frac{\snorm{\wvel(y,t)}^2}{\snorm{x-y}^2}\,\dy}^\half \Bp{\int_{\B^R}\snorm{\grad\wvel(y,t)}^2\,\dy}^\half\dt 
\leq \Cc{ae_asympexpansionthm}\iper\int_0^\per\int_{\B^R}\snorm{\grad\wvel(y,t)}^2\,\dy\dt, 
\end{aligned}
\end{align*}
whence by Corollary \ref{ae_firstcor} we obtain
\begin{align*}
\begin{aligned}
\forall \epsilon>0:\quad \snorml{\fundsoloseen*\proj\bb{\nsnonlin{\wvel}}(x)} \leq \Cc{ae_asympexpansionthm} R^{-1+\epsilon}. 
\end{aligned}
\end{align*}
In a similar manner we show
\begin{align*}
\begin{aligned}
\forall \epsilon>0:\quad \snorml{\fundsoloseen*\proj\bb{\nsnonlin{\vvel}}(x)} \leq \Cc{ae_asympexpansionthm} R^{-1+\epsilon}. 
\end{aligned}
\end{align*}
We conclude that 
\begin{align}
\begin{aligned}\label{ae_asympexpansionthmestofvvela}
\forall \epsilon>0:\quad \snorm{\vvel(x)} \leq \Cc{ae_asympexpansionthm} R^{-1+\epsilon}. 
\end{aligned}
\end{align}
Inserting \eqref{ae_asympexpansionthmestofvvela} into \eqref{ae_asympexpansionthmestofwvela}, we then find
\begin{align}\label{ae_asympexpansionthmEstwvelBeforeIteration}
\forall\epsilon>0:\quad \snorm{\wvel(t,x)}\leq \Cc{ae_asympexpansionthm}\, 
\bp{ R^{-\frac{2}{5}+\epsilon}\norm{\wvel}_{\LR{\infty}\np{\Torus\times\B^R}} + R^{-\frac{7}{5}+\epsilon}}.
\end{align}
By Lemma \ref{RegLem}, recall \eqref{RegLem_FinalLinftyRegwvel}, $\wvel\in\LR{\infty}(\grp)$.
We can thus iterate \eqref{ae_asympexpansionthmEstwvelBeforeIteration} a sufficient number of times to obtain
\begin{align}\label{ae_asympexpansionthmestofwvele}
\forall\epsilon>0:\quad \snorm{\wvel(t,x)}\leq \Cc{ae_asympexpansionthm}\, R^{-\frac{7}{5}+\epsilon}.
\end{align}
We now return to the representation formula \eqref{ae_asympexpansionthmrepofwvel} of $\wvel$. We shall utilize \eqref{ae_asympexpansionthmestofvvela}
and \eqref{ae_asympexpansionthmestofwvele} to extract even better decay estimates for $\wvel$. For this purpose, we recall \eqref{ae_asympexpansionthmSplitofwvel}
and observe that the integrability of $\vvel$, $\wvel$ and $\fundsolvelcompl$ established in Lemma \ref{RegLem} and Lemma \ref{ae_integrabilityoffundsolcompl},
respectively, allows us to integrate by parts and obtain
\begin{align}\label{ae_asympexpansionthmSplitofwvel_AfterIntByParts}
\begin{aligned}
&\bb{I_1(t,x)}_i=\partial_k\bb{\fundsolvelcompl}_{ij}*\bb{\wvel_j\wvel_k}(t,x),\\
&\bb{I_2(t,x)}_i=\partial_k\bb{\fundsolvelcompl}_{ij}*\bb{\vvel_j\wvel_k}(t,x),\\
&\bb{I_1(t,x)}_i=\partial_k\bb{\fundsolvelcompl}_{ij}*\bb{\wvel_j\vvel_k}(t,x).
\end{aligned}
\end{align} 
We again decompose
\begin{align*}
\begin{aligned}
\bb{I_1(t,x)}_i&=\iper\int_0^\per\int_{\B_R}\partial_k\bb{\fundsolvelcompl}_{ij}(x-y,t-s)\,\bb{\wvel_j\wvel_k}(y,s)\,\dy\ds\\
&\quad +\iper\int_0^\per\int_{\B^R}\partial_k\bb{\fundsolvelcompl}_{ij}(x-y,t-s)\,\bb{\wvel_j\wvel_k}(y,s)\,\dy\ds\\
&=: \tI_{11} + \tI_{12}.
\end{aligned}
\end{align*}
Employing H\"older's inequality and recalling \eqref{ae_estoffundsolcomplgradest}, we deduce 
\begin{align*}
\begin{aligned}
\snorm{\tI_{11}}&\leq  \int_{\B_R} \Bp{\iper\int_0^\per\snorml{\partial_k\bb{\fundsolvelcompl}_{ij}(x-y,s)}^2\,\ds}^\half \Bp{\iper\int_0^\per\snorm{\wvel}^4\,\ds}^\half \,\dy\\
&\leq \Cc{ae_asympexpansionthm} \frac{1}{R^4}\, \Bp{\int_{\B_R} 1\,\dy}^\half \norm{\wvel}_4^2 \leq \Cc{ae_asympexpansionthm} R^{-\frac{5}{2}},
\end{aligned}
\end{align*}
where we in the last inequality use that ${\wvel}\in\LR{4}(\grp)$, which is a direct consequence of Lemma \ref{RegLem}.
We recall from Lemma \ref{ae_integrabilityoffundsolcompl} that $\partial_k\fundsolvelcompl\in\LR{1}(\grp)$ and infer by \eqref{ae_asympexpansionthmestofwvele} the estimate
\begin{align*}
\begin{aligned}
\forall \epsilon>0:\quad \snorm{\tI_{12}}&\leq  \norm{\partial_k\fundsolvelcompl}_1\,\norm{\wvel}_{\LR{\infty}\np{\Torus\times\B^R}}^2
\leq \Cc{ae_asympexpansionthm} R^{-\frac{14}{5}+\epsilon}.
\end{aligned}
\end{align*}
Consequently,
\begin{align*}
\begin{aligned}
\forall \epsilon>0:\quad \snorm{I_{1}}&\leq \Cc{ae_asympexpansionthm}\bp{ R^{-\frac{5}{2}}+ R^{-\frac{14}{5}+\epsilon}}.
\end{aligned}
\end{align*}
In a similar manner, we estimate
\begin{align*}
\begin{aligned}
\forall \epsilon>0:\quad \snorm{I_{2}}+\snorm{I_{3}}&\leq \Cc{ae_asympexpansionthm}\bp{ R^{-\frac{5}{2}}+ R^{-\frac{12}{5}+\epsilon}}.
\end{aligned}
\end{align*}
Recalling \eqref{ae_asympexpansionthmestofI4final}, we may thus conclude
\begin{align}\label{ae_asympexpansionthmestofwvelFinal}
\forall\epsilon>0:\quad \snorm{\wvel(t,x)}\leq \Cc{ae_asympexpansionthm}\, \snorm{x}^{-\frac{12}{5}+\epsilon}.
\end{align}
We finally return to the representation formula \eqref{ae_asympexpansionthmRepofvvel} for $\vvel$. Recalling the integrability properties 
\eqref{ae_oseenfundsolprops_SummabilityFundSolGrad} and \eqref{ae_oseenfundsolprops_LocSummabilityFundSolGrad} of 
$\grad\fundsoloseen$, we estimate
\begin{align*}
\begin{aligned}
\snorml{\fundsoloseen*\proj\bb{\nsnonlin{\wvel}}(x)}
&= \snormL{\iper\int_0^\per\int_{\R^3} \fundsoloseen(x-y)\, \Div\bb{\wvel\otimes\wvel}(y,t)\,\dy\dt}\\
&\leq \Cc{ae_asympexpansionthm} \iper\int_0^\per\int_{\R^3} \snorml{\grad\fundsoloseen(x-y)}\, \snorm{\wvel(y,t)}^2\,\dy\dt.
\end{aligned}
\end{align*}
We then use \eqref{ae_fundsoloseengradpdecay} to estimate
\begin{align*}
\begin{aligned}
\iper\int_0^\per\int_{\B_R} \snorml{\grad\fundsoloseen(x-y)}\, \snorm{\wvel(y,t)}^2\,\dy\dt
&\leq \Cc{ae_asympexpansionthm} R^{-\frac{3}{2}}\,\iper\int_0^\per\int_{\B_R} \snorm{\wvel(y,t)}^2\,\dy\dt\\
&\leq \Cc{ae_asympexpansionthm} R^{-\frac{3}{2}},
\end{aligned}
\end{align*}
where we in the last inequality recall that $\wvel\in\LR{2}(\grp)$. Moreover, in view of \eqref{ae_fundsoloseenintgradpdecay}
and \eqref{ae_asympexpansionthmestofwvelFinal}, we see that 
\begin{align*}
\begin{aligned}
&\iper\int_0^\per\int_{\B_{3R,R}} \snorml{\grad\fundsoloseen(x-y)}\, \snorm{\wvel(y,t)}^2\,\dy\dt
\leq \Cc{ae_asympexpansionthm} R^{-\frac{24}{5}+\epsilon}\, \int_{\B_{3R,R}} \snorml{\grad\fundsoloseen(x-y)}\,\dy\\
&\qquad\leq \Cclast{ae_asympexpansionthm} R^{-\frac{24}{5}+\epsilon}\, \int_{\B_{6R}} \snorml{\grad\fundsoloseen(y)}\,\dy
\leq \Cc{ae_asympexpansionthm} R^{-\frac{24}{5}+\epsilon}\, \int_0^{6R}\int_{\partial\B_{r}} \snorml{\grad\fundsoloseen(y)}\,\dS\dr\\
&\qquad\leq \Cc{ae_asympexpansionthm} R^{-\frac{24}{5}+\epsilon}\, \int_0^{6R} r^{-\half}\,\dr
\leq \Cc{ae_asympexpansionthm} R^{-\frac{24}{5}+\half +\epsilon}
\end{aligned}
\end{align*}
for all $\epsilon>0$.
Finally, employing again \eqref{ae_fundsoloseengradpdecay} and the fact that $\wvel\in\LR{2}(\grp)$, we estimate
\begin{align*}
\begin{aligned}
\iper\int_0^\per\int_{\B^{3R}} \snorml{\grad\fundsoloseen(x-y)}\, \snorm{\wvel(y,t)}^2\,\dy\dt
\leq \Cc{ae_asympexpansionthm} R^{-\frac{3}{2}}.
\end{aligned}
\end{align*}
We thus conclude that
\begin{align*}
\begin{aligned}
\snorml{\fundsoloseen*\proj\bb{\nsnonlin{\wvel}}(x)}
\leq \Cc{ae_asympexpansionthm} \snorm{x}^{-\frac{3}{2}}.
\end{aligned}
\end{align*}
The other terms in the representation formula \eqref{ae_asympexpansionthmRepofvvel} for $\vvel$ also appear 
in the analogous representation formula for a solution to the classical, steady-state Navier-Stokes system.  
We can therefore estimate them using well-known methods. More specifically, in view of 
\eqref{ae_asympexpansionthmestofvvela} we can use the arguments from the proof of \cite[Theorem IX.8.1]{galdi:book2} to obtain
\begin{align*}
\begin{aligned}
\forall\epsilon>0:\quad \snorml{\fundsoloseen*\bb{\nsnonlin{\vvel}}(x)}
\leq \Cc{ae_asympexpansionthm} \snorm{x}^{-\frac{3}{2}+\epsilon}
\end{aligned}
\end{align*}
and
\begin{align*}
\begin{aligned}
\snormL{\fundsoloseen*\proj f(x)-\fundsoloseen(x)\cdot\bigg(\int_{\R^3}\proj f\bigg)}
\leq \Cc{ae_asympexpansionthm} \snorm{x}^{-\frac{3}{2}}.
\end{aligned}
\end{align*}
We therefore finally deduce, recalling \eqref{ae_asympexpansionthmRepofvvel} and \eqref{ae_asympexpansionthmestofwvelFinal}, that 
\begin{align*}
\begin{aligned}
\uvel(t,x) &= \vvel(x) + \wvel(t,x)\\
&= \fundsoloseen(x)\cdot\bigg(\int_{\R^3}\proj f\bigg) + \restterm(t,x)
\end{aligned}
\end{align*}
with $\restterm(t,x)$ satisfying \eqref{MainThm_estofrestterm}. 
\end{proof}


\bibliographystyle{abbrv}

\end{document}